\documentclass[a4paper, reqno]{amsart}
\usepackage{amssymb, amsmath, amsthm, bm, latexsym, enumerate, marginnote,  tikz}
\usepackage{float}
\usepackage{cases}
\usepackage[mathcal]{euscript}
\usepackage[colorlinks,
linkcolor=red,
anchorcolor=magenta,
citecolor=blue]{hyperref}
\usepackage{hyperref}
\usepackage{mathrsfs}
\usepackage{amscd}
\usepackage{mathtools}
\usepackage{amsbsy}
\usepackage{graphicx}
\usepackage{color}
\usepackage[all]{xy}
\usepackage{caption}
\usepackage{subcaption}
\usepackage{bigstrut}
\usepackage{geometry}
\begin{document}
	\newtheorem{theoreme}{Theorem}
	\newtheorem{lemma}{Lemma}[section]
	\newtheorem{proposition}[lemma]{Proposition}
	\newtheorem{corollary}[lemma]{Corollary}
	\newtheorem{definition}[lemma]{Definition}
	\newtheorem{conjecture}[lemma]{Conjecture}
	\newtheorem{remark}[lemma]{Remark}
	\newtheorem{exe}{Exercise}
	\newtheorem{theorem}[lemma]{Theorem}
	\theoremstyle{definition}
	\numberwithin{equation}{section}
	\newcommand{\ud}{\mathrm{d}}
	\newcommand{\ess}{\rm ess}
	\newcommand{\trans}{{\rm trans}}
	\newcommand{\R}{\mathbb R}
	\newcommand{\TT}{\mathbb T}
	\newcommand{\Z}{\mathbb Z}
	\newcommand{\N}{\mathbb N}
	\newcommand{\Q}{\mathbb Q}
	\newcommand{\Var}{\operatorname{Var}}
	\newcommand{\br}{{\rm BL}}
	\newcommand{\tr}{\operatorname{tr}}
	\newcommand{\supp}{\operatorname{Supp}}
	\newcommand{\intinf}{\int_{-\infty}^\infty}
	\newcommand{\me}{\mathrm{e}}
	\newcommand{\blue}[1]{{\color{blue}#1}}
	\newcommand{\red}[1]{{\color{red}#1}}
	\newcommand{\mi}{\mathrm{i}}
	\newcommand{\beq}{\begin{equation}}
		\newcommand{\eeq}{\end{equation}}
	\newcommand{\beqq}{\begin{equation*}}
		\newcommand{\eeqq}{\end{equation*}}
	\newcommand{\ben}{\begin{eqnarray}}
		\newcommand{\een}{\end{eqnarray}}
	\newcommand{\beno}{\begin{eqnarray*}}
		\newcommand{\eeno}{\end{eqnarray*}}
	\def \a{\rm A}
	\def \f{\frac}
	\def\a{\alpha}
	\def\e{\varepsilon}
	\def\ld{\lambda}
	\def\p{\partial}
	\def\v{\varphi}
	\newcommand{\D}{\Delta}
	\newcommand{\Ld}{\Lambda}
	\newcommand{\n}{\nabla}
	\newcommand{\ta}{\rm tang}
	\newcommand{\tra}{\rm trans}
	\newcommand{\an}{\rm Ang}
	\newcommand{\GG}{\text{g}}
	\newcommand{\tn}{\tilde{\nu}}
	\newcommand{\tw}{\tilde{w}}
	\newcommand{\tl}{\tilde{\ell}}

	\title[Improved local smoothing estimates for the wave equation]
	{Improved local smoothing estimate for the wave equation in higher dimensions}

	\subjclass[2010]{Primary:35S30; Secondary: 35L15}
	
	\keywords{Local smoothing; wave equation; $k$-broad ``norm"}

	\begin{abstract}
		In this paper, we establish the sharp $k$-broad estimate for a class of  phase functions satisfying the homogeneous convex conditions. As an application, we obtain improved local smoothing estimates for the half-wave operator in dimensions $n\ge3$. As a byproduct, we also generalize the restriction estimates of Ou--Wang \cite{OW} to a broader class of phase functions.
	\end{abstract}
	\author{Chuanwei Gao}
	\address{School of Mathematical Sciences, Capital Normal University, Beijing 100048, China}
	\email{cwgao@cnu.edu.cn}

	\author{Bochen Liu}
	\address{Department of Mathematics \& International Center for Mathematics, Southern University of Science and Technology, Shenzhen 518055, China}
	\email{Bochen.Liu1989@gmail.com}
	
	\author{Changxing Miao}
	\address{Institute for Applied Physics and Computational Mathematics, Beijing, China}
	\email{miao\textunderscore changxing@iapcm.ac.cn}
	\author{Yakun Xi}
	\address{School of Mathematical Sciences, Zhejiang University, Hangzhou 310027, PR China}
	\email{yakunxi@zju.edu.cn}
	
	\maketitle

	\section{introduction}\label{section-1}

	Let $u$ be the solution to the Cauchy problem
	\begin{equation}\label{eq-01}
		\begin{cases}
			(\partial_{tt}-\Delta)u=0,\;(t,x)\in \R\times \R^n,\\
			u(0,x)=f,\ \partial_tu(0,x)=0
		\end{cases}
	\end{equation}
	where $f$ is a Schwartz function. $u$ can be expressed in terms of the half-wave operator
	$e^{it\sqrt{-\Delta}}$ as
	\begin{equation*}
		u(x,t)=\frac12\Big(e^{it\sqrt{-\Delta}}f+e^{-it\sqrt{-\Delta}}f\Big).
	\end{equation*}
	This paper is concerned with the  $L^p$-regularity estimate of the solution $u$. For fixed time $t$,  the classical sharp estimate of Peral \cite{Peral} and Miyachi \cite{Miya} reads:
	\beq \label{eq-02}
	\|e^{it\sqrt{-\Delta}}f\|_{L^p(\R^n)}\leq C_{t,p}\|f\|_{L^p_{s_p}},\;\;\;\; s_p:=(n-1)\Big|\frac{1}{2}-\frac{1}{p}\Big|, \;\; 1<p<\infty.
	\eeq
	This estimate trivially leads to the following space-time estimate
	\begin{equation}\label{eq-03}
		\Big(\int_1^2 \| e^{it\sqrt{-\Delta}}f\|_{L^p(\R^{n})}^p \,\ud t \Big)^{1/p} \lesssim \| f \|_{L^p_{s_p}(\R^n)}.
	\end{equation}
	One natural question then arises: can one do better than \eqref{eq-03}?  More precisely, does there exist some $\varepsilon >0$ such that \eqref{eq-03} holds with $s_p-\varepsilon$ in place of $s_p$? The following  local smoothing conjecture was formulated by Sogge \cite{Sogge91}.
	
	\begin{conjecture}[Local smoothing conjecture]\label{LS conj euclidean}
		For $n \geq 2$, the inequality
		\begin{equation}\label{eq-05}
			\Big( \int_1^2 \|e^{i t \sqrt{-\Delta}} f \|^p_{L^p(\mathbb{R}^n )} \ud t \Big)^{\frac1p} \lesssim \|f\|_{L^p_{s_p-\sigma}(\R^n)}
		\end{equation}
		holds for all
		\begin{equation}\label{eq-06}\sigma<\left\{\begin{aligned}
				& \tfrac1p, \;\;\;\text{\rm if}\;\; \tfrac{2n}{n-1} < p < \infty;\\
				&s_p,\;\; \text{\rm if }\;\;\;2 < p \leq \tfrac{2n}{n-1}.\end{aligned}\right.
		\end{equation}
	\end{conjecture}

	In the same paper, Sogge \cite{Sogge91} obtained the first partial results on the above conjecture for all $p>2$ when $n=2$, which were greatly simplified and further improved in his joint work with Mockenhaupt and Seeger \cite{MSS}, where the {\it square function} approach was introduced.
	In 2000, Wolff \cite{Wolff00} proved Conjecture \ref{LS conj euclidean} for the case $n=2$ and $p>74$ by introducing what is now known as the {\it decoupling inequality} for the cone. Following Wolff, decoupling inequalities have been studied by many authors \cite{GaSchSe,GaSe09, LaWo02}. In 2015,  Bourgain--Demeter  \cite{BoDe2015} established the full range of sharp $\ell^2$-decoupling inequalities in all dimensions, of which the influence permeates into number theory,  PDEs and geometric measure theory. As a direct consequence, 
	Bourgain and Demeter obtained the sharp local smoothing estimate for all  $n\geq 2$ and $p\geq \tfrac{2(n+1)}{n-1}$.
	Recently, Guth-Wang-Zhang \cite{GWZ} resolved  the local smoothing conjecture for  $n=2$ by establishing the full range sharp square function inequality. The purpose of this paper is to further improve the local smoothing result for dimensions $n\geq 3$ and $2<p< \frac{2(n+1)}{n-1}$. In particular, we obtain
	\begin{theorem}\label{theoa2}
		Let $n\ge3$ and
		\begin{equation}\label{eq:maina}p>\left\{\begin{aligned}
				&2\tfrac{3n+5}{3n+1}\quad \text{\rm for $n$ odd},\\
				&2\tfrac{3n+6}{3n+2}\quad \text{\rm for $n$ even}.
			\end{aligned}\right.\end{equation}
		Then
		\beq \label{eq-07}
		\big\|e^{it\sqrt{-\Delta}}f\big\|_{L^p(\R^n\times [1,2])}\leq C\|f\|_{L^p_{s_p-\sigma}}, \;\text{for all}\;\sigma<\tfrac2p-\tfrac12.
		\eeq
	\end{theorem}
	
	For $n\geq 3$ and $2< p<\frac{2(n+1)}{n-1}$ , the sharp $\ell^p$ decoupling inequality of Bourgain--Demeter \cite{BoDe2015} and the  $L^2$ energy estimate implies
	\beq\label{eq:aa1}
	\big\|e^{it\sqrt{-\Delta}}f\big\|_{L^p(\R^n\times [1,2])}\leq C\|f\|_{L^p_{s_p-\sigma}},\quad\; \sigma<\tfrac{n-1}{2}\big|\tfrac{1}{2}-\tfrac{1}{p}\big|.
	\eeq
	A direct calculation shows that if  $p<\frac{2(n+3)}{n+1}$,
	\beqq
	\frac{n-1}{2}\Big|\frac{1}{2}-\frac{1}{p}\Big|<\frac{2}{p}-\frac{1}{2}.
	\eeqq
	One can see that Corollary \ref{coraaa}  below improves the previous best known local smoothing estimate \eqref{eq:aa1} in range of $2<p<\frac{2(n+1)}{n-1}$ for $n\geq 3$. Indeed, by interpolating using \eqref{eq-07}, \eqref{eq:aa1} and the trivial $L^2$ bound, we have the following.
	\begin{corollary}\label{coraaa}
		Let $n\geq 3$,  Then
		\beq 
		\big\|e^{it\sqrt{-\Delta}}f\big\|_{L^p(\R^n\times [1,2])}\leq C\|f\|_{L^p_{s_p-\sigma}}
		\eeq
		for $\sigma<\sigma_p$, where 
		if $n\geq 3$ is odd,  \begin{equation}\sigma_p=\left\{\begin{aligned}
				&\f{3n-3}{4}\big(\f{1}{2}-\f{1}{p}\big),\quad 2<p\leq 2\f{3n+5}{3n+1},\\
				&\f{n-1}{n+3}\big(\f{3n+1}{6n+10}-\f{1}{p}\big)+\f{3n-3}{6n+10},\quad 2\f{3n+5}{3n+1}<p\le2\f{n+1}{n-1},
			\end{aligned}\right.\end{equation}
		if $n \geq 3$ is even,  
		\begin{equation}\sigma_p=\left\{\begin{aligned}
				&\f{3n-2}{4}\big(\f{1}{2}-\f{1}{p}\big),\quad 2<p\leq 2\f{3n+6}{3n+2},\\
				&\f{n-2}{n+4}\big(\f{3n+2}{6n+12}-\f{1}{p}\big)+\f{3n-2}{6n+12},\quad 2\f{3n+6}{3n+2}<p\le2\f{n+1}{n-1}.
			\end{aligned}\right.\end{equation}

	\end{corollary}
	See Table \ref{table:1} for a detailed comparison for the improvement at the conjectured critical exponent $p_c=\frac{2n}{n-1}$ for $n=3,4,5,6$. See Figure \ref{fig1} for a $\sigma$-$p^{-1}$ plot for the odd $n$ case.	 \begin{figure}
		\centering
		\includegraphics[width=0.95\textwidth]{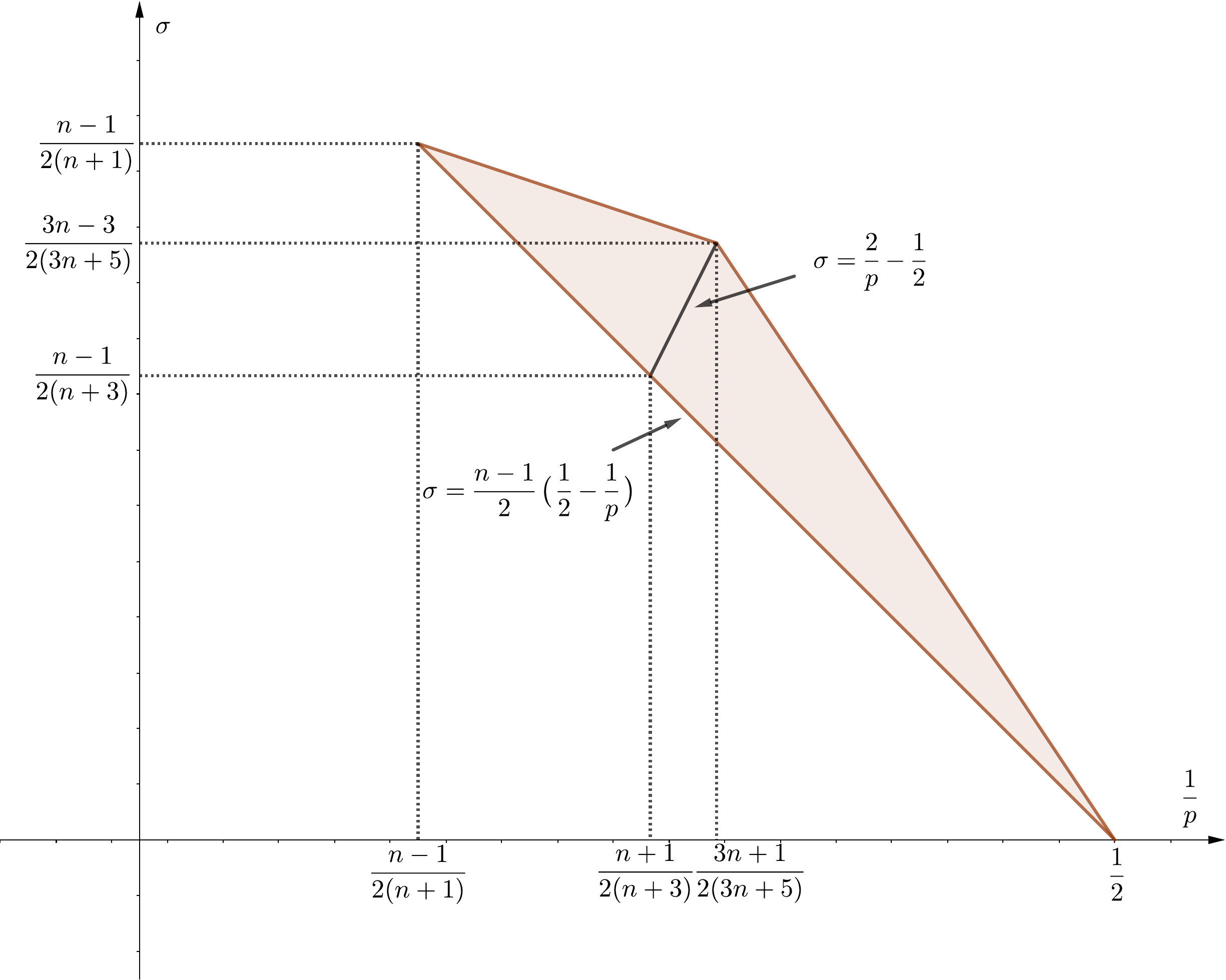}
		\caption{$\sigma$-$p^{-1}$ graph for Corollary \ref{coraaa} when $n$ is odd.}
		\label{fig1}
	\end{figure}
	\begin{table}
		\begin{center}
			\begin{tabular}{||c| c| c| c| c||} 
				\hline
				$n$ & $p_c$ & $\sigma_{p_c}$ (conjectured) &  $\sigma_{p_c}$ (\cite{BoDe2015})& $\sigma_{p_c}$ (Corollary \ref{coraaa}) \\ [0.5ex] 
				\hline
				3&$3$&$1/3$& $1/6$& $2/9$\\
				\hline
				4 & $8/3$ & $3/8$   &  $3/16$ & 9/32\\
				\hline
				5 & $5/2$ & $2/5$  & $1/5$ & $3/10$\\
				\hline
				6 & $12/5$ & $5/12$  & $5/24$ & $1/3$\\
				
				\hline
			\end{tabular}
			
			\caption{Comparing Corollary \ref{coraaa} to previous records at $p_c$.}
			\label{table:1}
		\end{center}
	\end{table}
	


	Local smoothing of the half wave operator has been studied extensively.  As discussed above, instead of handling the half wave operator $e^{it\sqrt{-\Delta}}f$ directly, people usually opt to establish decoupling inequalities or square function estimates, and then apply them to the local smoothing problem. At this point, let us briefly review both approaches.
	
	Define the domain $\Gamma\subset \R^n$ as
	\beqq
	\Gamma:=\big\{(\xi',\xi_n)\in \R^n\backslash\{0\}:\; 1\leq \xi_n\leq 2,\;\; |\xi'|\leq \xi_n\big\},
	\eeqq
	and make slab-decomposition with respect to $\Gamma$ in the following way.
	Assuming $R>1$, we select a collection of $R^{-\frac12}$-maximally separated points $\{(\xi'_\nu,1)\}_\nu$ in the unit ball $B^{n-1}(0,1)\times \{1\}$ of the affine hyperplane $\xi_n=1$. For each $\nu$, we define the $\nu$-slab as
	\beqq
	\nu:=\Big\{(\xi',\xi_n)\in \Gamma: \Big|\tfrac{\xi'}{\xi_n}-\xi_{\nu}\Big|\leq R^{-\frac12}\Big\},
	\eeqq
	Let  $\chi_\nu$ be the characteristic function of the $\nu$-plate, and  set  $f^\nu=\mathcal F^{-1}(\widehat{f}\chi_\nu)$.
	
	Under notation above, the $\ell^p$-decoupling inequality due to Bourgain--Demeter is 
	\beq\label{decou}
	\Big\|\sum_{\nu} e^{it{\sqrt{-\Delta}}}f^\nu\Big\|_{L^p(\R^{n+1})}\leq C_\varepsilon R^{(n-1)|\f{1}{2}-\f{1}{p}|-\frac{1}{p}+\varepsilon}\Big(\sum_{\nu} \|e^{it\sqrt{-\Delta}}f^\nu\|_{L^p(\R^{n+1})}^p\Big)^{\frac{1}{p}},
	\eeq
	for $ p\geq \tfrac{2(n+1)}{n-1}$. As a direct consequence of \eqref{decou}, Conjecture \ref{LS conj euclidean} has been resolved for the range $p\geq \frac{2(n+1)}{n-1}$.  It seems that $\ell^p$-decoupling inequality is well suited for handling local smoothing estimate with larger exponents, whereas,  inefficient for tackling  the case when $p$ is close to the endpoint $p=\frac{2n}{n-1}$.  In contrast, the following conjectured (reverse) square function inequality has been proven to be powerful near the endpoint.
	\beq \label{eq:211}
	\Big\|\sum_{\nu} e^{it{\sqrt{-\Delta}}}f^\nu\Big\|_{L^p(\R^{n+1})}\leq C_\varepsilon R^{\varepsilon} \Big\|\Big(\sum_\nu |e^{it\sqrt{-\Delta}}f^\nu|^2\Big)^{\frac{1}{2}}\Big\|_{L^p(\R^{n+1})},\;\;\; 2\leq p\leq \tfrac{2n}{n-1}.
	\eeq
	Recently, in the remarkable work of Guth--Wang--Zhang \cite{GWZ}, the authors established inequality \eqref{eq:211}
	in dimension two. Following the argument in \cite{MSS}, \eqref{eq:211} leads to the full range of sharp local smoothing estimate for $n=2$. 
	To be more precise,  to handle the local smoothing problem using the square function inequality, we also need to use the  Nikodym maximal function inequality associated to the cone.
	The sharp result for such Nikodym maximal function inequality was established in \cite{MSS} for $n=2$, but its higher dimensional counterpart is still wide open, which limits us, to some extent, to advance the research of the local smoothing in higher dimensions. For more discussion about the local smoothing estimate of the half-wave operator, see \cite{LeeJ,LV,HNS,Tao-Vargas-II}.
	
	In this paper, motivated by the seminal work of Guth \cite{Guth}, we circumvent these problems through handling the operator $e^{it\sqrt{-\Delta}}$ directly by employing the so-called $k$-broad ``norm" estimate, which can be seen as a weaker version of the multilinear restriction estimate due to Bennett--Carbery--Tao \cite{BCT06}. It is worth noting that, there is a difference in the results of Theorem \ref{theoa2}  between odd and even spatial dimensions. This is a common theme in the study of problems related to the restriction conjecture. In particular, for such problems in the variable coefficient setting, certain Kakeya compression phenomena exist. Such phenomena are usually different between odd and even dimensions. One may refer to \cite{BelHicSog18P-surv,GHI,MiSo,wisewell} for more details.
	Therefore, even though we only work in the Euclidean case in this article, we believe that the methods used in this paper may stimulate the research of the local smoothing estimate for the class of Fourier integral operator satisfying cinematic curvature conditions in higher dimensions. 
	
	Let us describe the outline of our proof and the key difficulties that arise. After some routine reductions, we shall perform a multi-scale broad-narrow argument which is inspired by the arguments in \cite{Guth, OW}. However, there are two main difficulties that we need to overcome. Firstly, unlike the restriction problem for a circular cone, Lorentz rescaling arguments are much more complicated in the local smoothing setting. As a result, we have to prove the $k$-broad estimate for general positively curved cones, which in turn requires us to establish several geometric lemmas without the nice symmetry of a circular cone. This is done in Section \ref{section-4}. Secondly, since we have to deal with general cones, we would need a narrow decoupling theorem for them. Unfortunately, however, unlike the parabolic case, the missing narrow decoupling theorem for general cones does not follow directly from the circular cone case via a Pramanik--Seeger approximation argument \cite{PS}. To overcome this, instead of proving estimates for a fixed class of cones, we opt for a new induction on scales argument with respect to a whole family of classes of cones $\mathbf\Phi(R)$, which are indexed by the physical scale $R$. To be more precise, the class of cones that we care about will approximate the circular cone better and better as the scale $R$ grows. The definition of the class $\mathbf\Phi(R)$ will be given in Section \ref{section-2}, and the induction argument will be given in Section \ref{section-7}. The authors believe that this approach is novel and may serve an important role in the study of other related problems.
	
	The rest of this paper is organized as follows: In Section \ref{section-2}, we review some preliminaries and basic reductions. In Section \ref{wavee}, we present the wave packet decomposition used in our proof. In Section \ref{section-4}, we establish a geometric lemma for a general class of cones, which plays a crucial role in the proof of $k$-broad ``norm'' estimate. In Section \ref{section-5}, we prove the $k$-broad ``norm'' estimate via polynomial partitioning, in the spirit of \cite{Guth, OW}. In Section \ref{paralemma}, a parabolic rescaling lemma suited for our setting will be established, which is a critical ingredient in our induction on scale argument. This parabolic rescaling lemma is similar to the ones established in \cite{BHS, GLMX} as we are dealing with a whole class of phase functions. Finally, we give the proof of Theorem \ref{theoa2} and 
	state the restriction estimates for general cones in Section \ref{section-7}.
	
	\indent{\bf Acknowledgements.} This project was supported by the National key R\&D program of China: No. 2022YFA1005700 and 2022YFA1007200. C. Gao was partially supported by Chinese Postdoc Foundation Grant: No. 8206300279. B. Liu was partially supported by SUSTech start-up counterpart Y01286235. C. Miao was partially supported by NSF of China grant: No. 11831004.  Y. Xi was partially supported by NSF of China grant: No. 12171424 and the Fundamental Research Funds for the Central Universities 2021QNA3001. The authors would like to thank Prof. Hong Wang for some helpful discussions.  The first author would like to thank Prof. Gang Tian's support and Prof. Dongyi Wei's helpful discussion. The authors want to acknowledge that David Beltran and Olli Saari realised independently \cite{BS} that it is possible to use $k$-broad  estimates to obtain improved local smoothing estimates. The authors want to thank Prof. David Beltran for some friendly communications, and for pointing out that broad and narrow bounds for the circular cone only is not enough for proving local smoothing type estimates. The authors want to thank an anonymous
	referee for his or her thorough and invaluable feedback.
	
	\indent{\bf Notation.} For non-negative quantities $X$ and $Y$, we will write $X\lesssim Y$ to denote the inequality $X\leq C Y$ for some constant $C>0$. If $X\lesssim Y\lesssim X$, we will write $X\sim Y$. Dependence of the implicit constants on the spatial dimensions or integral exponents such as $p$ will be suppressed; dependence on additional parameters will be indicated using subscripts or parenthesis. For example, $X\lesssim_u Y$ indicates $X\leq CY$ for some $C=C_u$. For a  function $A(R)$, we write $A(R)={\rm RapDec}(R)$ if for any $N\in\mathbb N$, there is a constant $C_N$ such that
	\beqq
	| A(R)|\leq C_N R^{-N} \,\,\quad\text{for all}\,\; R\geq 1.
	\eeqq
	
	Throughout the paper, $\chi_E$ denotes the characteristic function of the set $E$.
	We usually denote by $B_R^n(a)$, or simply $B_R(a)$, a ball in $\R^n$ with center $a$ and radius $R$. We will also denote by $B_R^n$, or simply $B_R$, a ball of radius $R$ and arbitrary center in $\R^n$. Let $r>0$, for the sake of convenience, we denote $C_{r}^{n+1}$ to be the cylinder $B_r^{n}\times [-r,r]$. Denote by $A(R):=B_{2R}^n(0)\setminus B_{R/2}^n (0)$. We denote $w_{B^{n}_{R}(x_0)}$ to be a nonnegative weight function adapted to the ball $B^{n}_{R}(x_0)$ such that
	$$ w_{B^{n}_{R}(x_0)}(x)\lesssim (1+R^{-1}|x-x_0|)^{-M},$$
	for some large constant $M\in \mathbb{N}$. 
	
	For any subspace $V\subset \R^n$ and $\eta\in \R^n$, we adopt the notation ${\rm Ang}(\eta,V)$ to denote the smallest angle between $\eta$ and any given vector $v\in V\backslash \{0\}$. Let $W\subset \R^n$ be another subspace, define ${\rm Ang}(V,W)$ as
	\beqq
	{\rm Ang}(V,W):=\min_{v\in V\backslash\{0\}, w\in W\backslash \{0\}}{\rm Ang}(v,w).
	\eeqq


	\section{preliminaries }\label{section-2}
	\subsection{Basic reductions and phase function classes} In this paper, as is standard in rescaling arguments, we shall prove estimates associated to a large scale $R\gg 1$, and an arbitrarily small parameter $\varepsilon>0$. We say a phase function $\phi$ is (positively) homogeneous of degree one if it  satisfies 
	\begin{itemize}
		\item[${\bf H_1}$:]$ \phi\in C^\infty(\R^n\backslash\{0\}),\  \phi(\lambda \xi)=\lambda \phi(\xi), \forall \lambda>0$.
	\end{itemize}
	Moreover, we say $\phi$	 satisfies the \emph{homogeneous convex} conditions if it satisfies the following condition as well.
	\begin{itemize}
		\item[${\bf H_2}$:]The Hessian of $\phi$ i.e. $\big(\tfrac{\partial^2\phi}{\partial x_i\partial x_j}\big)_{n\times n}$
		has $(n-1)$-positive eigenvalues.
	\end{itemize}
	A prototypical example for the homogeneous convex function is given by $\phi(\xi)=|\xi|$.  We are concerned with the half-wave operator $e^{it\sqrt{-\Delta}}$, and it can be reduced to considering a oscillatory integral  involving the phase function $|\xi|$. For technical reasons, we need to employ an induction on scales argument which requires the phase function to stay invariant under certain transformations, while the function $|\xi|$ alone does not ensure this. Thus, we shall work with the following class of phase functions.
	\begin{definition}[Phase function class $\mathbf\Phi$]We say a function $\phi$ lies in the class $\mathbf\Phi$, if $\phi$ obeys the homogeneous convex conditions ${\bf H_1}$, ${\bf H_2}$, with eigenvalues of 
		$\big(\frac{\partial^2\phi}{\partial \xi_i\partial \xi_j}\big)$ in the interval $[1/2, 2]$, and 
		\beq \label{assu}|\partial^\alpha\phi(\xi)|\leq C_{\rm par}, \;\; \forall\,|\alpha|\leq N_{\rm par},\;\; \xi\in {\rm N}_{\varepsilon_0}(e_n),\eeq
		where $C_{\rm par}>0$, $0<\varepsilon_0\ll 1$ and $N_{\rm par}\in\mathbb N$ are universal constants, and ${\rm N}_{\varepsilon_0}(e_n)$ denotes the $\varepsilon_0$-neighborhood of $e_n$.
	\end{definition}
	
	Moreover, to facilitate the proof of the narrow decoupling theorem in Section \ref{section-7}, we will also consider a more special class of phase functions satisfying a more precised condition which we will define now. Let $K_0>0,\,\tilde\delta=\tilde\delta(\varepsilon)>0$, both of which will be chosen later in the argument.
	
	\begin{itemize}
		\item[${\bf H_3}$:]Let $K=K_0 R^{\tilde{\delta}}$, and
		$$\phi_R(\xi)=\frac{\xi_1^2+\cdots+\xi_{n-1}^2}{2\xi_n}+K^{-4}{\rm E}_R(\xi),\quad\forall \xi\in {\rm N}_{\varepsilon_0}(e_n),$$
		where ${\rm E}_R(\xi)$ is a homogeneous function of  degree $1$ and satisfies 
		$$|\partial^\alpha {\rm E}_R(\xi)|\leq c_{\rm par},\; \; |\alpha|\leq N_{\rm par},$$
		for some fixed constant $0<c_{\rm par}\ll 1$.
	\end{itemize}
	\begin{definition}[Phase function class $\mathbf\Phi(R)$]We say a function $\phi_R$ is in the class $\mathbf\Phi(R)$, if $\phi_R$ obeys condition ${\bf H_3}$.
	\end{definition}
	Here and throughout the paper, we shall always assume  $\phi\in\mathbf\Phi$ and $\phi_R\in\mathbf\Phi(R)$. It should be noted that one needs to be extra careful when working with the class  $\mathbf\Phi(R)$, since it depends on the scale $R$. To be more precise, we need to make sure that after each rescaling step, our new phase function lands in the appropriate class associated to the new scale.  In addition, since it is easy to check that $\mathbf\Phi(R)\subset\mathbf\Phi$, any statements that we prove for phase functions in the bigger class $\mathbf\Phi$ will certainly hold for any $\phi_R\in\mathbf\Phi(R)$.
	
	Next, let us collect some useful standard results from previous works.
	
	\subsection{Transference between local and global estimates}
	This reduction was used in \cite{GMZ} for a slightly different case. We give the details here for completeness.
	
	Let  $\psi$ be a non-negative smooth function on $\R^n$ such that
	\beq\label{poss}
	{\rm supp}\;\widehat{\psi} \subset B_1^n(0),\;\sum_{\ell \in \mathbb{Z}^{n}} \psi(x-\ell)\equiv 1,\quad\; \forall\; x\in \R^n.
	\eeq
	Define  $\psi_{\ell}(x):= \psi(R^{-1}x-\ell)$ and  $f_\ell=\psi_\ell f$.
	\begin{lemma}\label{lemma-1}  Assume ${\rm supp}\,\widehat{f}\subset {\rm A}(1)$, then for any $\varepsilon>0$, there holds
		\beq\label{eq:2}
		|e^{it\sqrt{-\Delta}} f(x)|\lesssim_{\varepsilon}  \big|e^{it\sqrt{-\Delta}} \big(\Psi_{B_{R^{1+\varepsilon}}^n(x_0)}f\big)(x)\big|+{\rm RapDec}(R)\sum_{|\ell|>R^\varepsilon} \big\|f|\psi_\ell(\cdot-x_0)|^{\f{1}{2}}\big\|_{L^p(w_{B_{R}^n(x_0)})},
		\eeq
		for $(x,t)\in B_{R}^n(x_0)\times [-R,R]$, $1< p<\infty$, where
		\begin{align*}
			\Psi_{B_{R^{1+\varepsilon}}^n(x_0)}(x):=&\sum_{|\ell |\leq  R^\varepsilon}\psi(R^{-1}(x-x_0)-\ell).
		\end{align*}
		
	\end{lemma}
	\begin{proof}
		Without loss of generality, we may assume that $x_0=0$.  
		We rewrite $e^{it\sqrt{-\Delta}} f$ via \eqref{poss} as
		\beq\label{equ:eitadd}
		e^{it\sqrt{-\Delta}} f(x)=\sum_{\ell \in \mathbb{Z}^{n}} \int_{\R^n}\int_{\R^n}e^{i((x-y)\cdot \xi+t|\xi|)}\eta(\xi)f_\ell (y)\, d\xi d y,
		\eeq
		where $\eta(\xi)\in C_c^{\infty}(B_1^n(0))$ satisfying that $\eta(\xi)=1$ for $\xi\in B_1^n(0)$. The associated kernel  $K_t(\cdot)$ of the operator $e^{it\sqrt{-\Delta}}\eta(D)$ is given by 
		\beqq
		K_t(x)=\int_{\R^n} e^{i(x\cdot\xi+t|\xi|)}\eta(\xi)\,d \xi.
		\eeqq
		Noting that $|t|\leq R$, by an integration by parts argument, we see that
		\beq\label{eq:e2}
		|K_t(x)|\leq C\chi_{|x|\leq CR}+ C_M\frac{\chi_{|x|\geq CR}}{(1+|x|)^M},\quad \text{ for all } M\in \mathbb N.
		\eeq
		Now we decompose $e^{it\sqrt{-\Delta}}f(x)$ into two parts
		\begin{align}
			e^{it\sqrt{-\Delta}}f(x)&= \sum_{|\ell |\leq R^\varepsilon} e^{it\sqrt{-\Delta}}f_\ell(x)+\sum_{|\ell|> R^\varepsilon } e^{it\sqrt{-\Delta}}f_\ell(x)\nonumber\\
			&=e^{it\sqrt{-\Delta}}(\Psi_{B_{R^{1+\varepsilon}}^n(0)}f)(x)+\sum_{|\ell|> R^\varepsilon } e^{it\sqrt{-\Delta}}f_\ell(x).\label{eq:5}
		\end{align}
		To complete the proof it suffices to bound the second term.  By H\"older's inequality, we have
		\begin{align*}
			\Big|\sum_{|\ell|> R^\varepsilon } e^{it\sqrt{-\Delta}}f_\ell(x)\Big|=&\Big|\sum_{|\ell|>R^\varepsilon} \int_{\R^n} K_t(x-y)f_\ell(y)\,d y\Big|\\
			\le &\Big|\sum_{|\ell|>R^\varepsilon}\int_{\R^n} |K_t(x-y)|^{\f{1}{2}}|\psi_\ell(y)|^{\f{1}{2}}|\psi_\ell(y)|^{\f{1}{2}}|f(y)||K_t(x-y)|^{\f{1}{2}}\,d y\Big|\\
			\le &\sum_{|\ell|>R^\varepsilon}\Big(\int_{\R^n} |K_t(x-y)|^{\f{p'}{2}}|\psi_\ell(y)|^{\f{p'}{2}}\,d y\Big)^{\f{1}{p'}}\Big(\int_{\R^n} |\psi_\ell(y)|^{\f{p}{2}}|f(y)|^p|K_t(x-y)|^{\f{p}{2}}\,d y\Big)^{\f{1}{p}}.
		\end{align*}
		For $(x,t)\in B_{R}^{n}(0)\times [-R,R]$, using the rapidly decay of $\psi$ and \eqref{eq:e2}, we have
		\begin{align*}
			|K_t(x-y)\psi_\ell(y)|\lesssim_M\f{R^{-\varepsilon M}}{\big(1+|R^{-1}y-\ell|\big)^M}, \quad\; |\ell|>R^\varepsilon,\; \forall\;x\in B_{R}^n(0),\;y\in \R^n,
		\end{align*}
		and
		$$|K_t(x-y)|\lesssim_M  \f{1}{\Big(1+\f{|y|}{R}\Big)^{M/2}}, \;\; \forall \;x\in B_{R}^n(0),\;y\in \R^n.$$
		Hence,
		\begin{align*}
			\Big|\sum_{|\ell|> R^\varepsilon } e^{it\sqrt{-\Delta}}f_\ell(x)\Big|\lesssim_M R^{-\varepsilon M+\frac{n}{p'}}\sum_{|\ell|>R^\varepsilon}
			\big\|f|\psi_\ell|^{\f{1}{2}}\big\|_{L^p(w_{B_{R}^n(0)})}.
		\end{align*}
	\end{proof}

	As a immediate consequence  of Lemma \ref{lemma-1}, we obtain the relation between local and global estimates in the spatial variables.
	\begin{corollary}\label{lemma-2}
		Let  $I\subset [-R,R]$ be  an interval. Suppose  ${\rm supp}\,\widehat{f}\subset {\rm A}(1)$ and
		\beq\label{eq:7}
		\|e^{it\sqrt{-\Delta}}f\|_{L_{x,t}^p(B_{R}^n \times I)}\leq C R^s \|f\|_{L^p},
		\eeq
		then, given any $ \varepsilon>0$, there exists a constant $C_\varepsilon$ such that
		\beq\label{eq:3}
		\|e^{it\sqrt{-\Delta}}f\|_{L_{x,t}^p( \R^n\times I )}\le C_\varepsilon R^{s+\varepsilon} \|f\|_{L^p}.
		\eeq
	\end{corollary}
	
	\begin{proof}
		Let $\{B_{R}^n(x_k)\}_{k\in \mathbb{Z}^n}$ be a covering of $\R^n$ using balls of radius $R$ with bounded overlaps. We have
		\beqq
		\big\|e^{it\sqrt{-\Delta}}f\big\|_{L_{x,t}^p(\R^n\times I)}^p \leq \sum_{k\in \mathbb Z^n}  \big\|e^{it\sqrt{-\Delta}}f\big\|_{L_{x,t}^p(B_{R}^n(x_k) \times I)} ^p.
		\eeqq
		Using  Lemma \ref{lemma-1}, we get
		\begin{align*}\label{eq:6}
			\big\|e^{it\sqrt{-\Delta}}f\big\|_{L_{x,t}^p ( B_{R}^n(x_k) \times I)}\lesssim_\varepsilon & \big\|e^{it\sqrt{-\Delta}}\big(\Psi_{B_{R^{1+\varepsilon/10n}}^n(x_k)}f\big)\big\|_{L_{x,t}^p (B_{R}^n(x_k) \times I)}\\
			&+{\rm RapDec}(R)\sum_{|\ell|>R^\varepsilon} \|f|\psi_\ell(\cdot-x_k)|^{\f{1}{2}}\|_{L^p(w_{B_{R}^n(x_k)})}.
		\end{align*}

		Summing over $k$, we obtain
		\begin{align}
			\sum_k  \|e^{it\sqrt{-\Delta}}f\|_{L_{x,t}^p (B_{R}^n(x_k) \times I)}^p\lesssim_\varepsilon&  \sum_k \Big(\big\|e^{it\sqrt{-\Delta}}\big(\Psi_{B_{R^{1+\varepsilon/10n}}^n(x_k)}f\big)\big\|_{L_{x,t}^p (B_{R}^n(x_k)\times I)}\Big)^p\\
			&+{\rm RapDec}(R)\sum_k \Big(\sum_{|\ell|>R^\varepsilon} \big\|f|\psi_\ell(\cdot-x_k)|^{\f{1}{2}}\big\|_{L^p(w_{B_{R}^n(x_k)})}\Big)^p.
			\label{eq:61}\end{align}
		It follows from \eqref{eq:7} and the bounded overlaps of $\{B_{R}^n(x_k)\}_k$ that the first term can be estimated by
		\beqq
		\sum_k \Big(\big\|e^{it\sqrt{-\Delta}}\big(\Psi_{B_{R^{1+\varepsilon/10n}}^n(x_k)}f\big)\big\|_{L_{x,t}^p(B_{R}^n(x_k)\times I)}\Big)^p\lesssim R^{sp+\varepsilon p}\|f\|_p^p.
		\eeqq
		To finish the proof, we use Minkowski's inequality to see that the second term satisfies \begin{equation*}
			{\rm RapDec}(R)\sum_k \Big(\sum_{|\ell|>R^\varepsilon} \|f|\psi_\ell(\cdot-x_k)|^{\f{1}{2}}\|_{L^p(w_{B_{R}^n(x_k)})}\Big)^p
			\lesssim{\rm RapDec}(R) \|f\|_{L^p}^p.
		\end{equation*}
	\end{proof}
	
	\subsection{Reducing to the class $\mathbf \Phi(R)$} 
	To prove Theorem \ref{theoa2}, it suffices to show
	\begin{align}\label{eq:301}
		\|e^{it\sqrt{-\Delta}}f\|_{L^p(B^{n}_{R}\times [-R,R])}\leq C_\varepsilon R^{n(\frac{1}{2}-\frac{1}{p})+\varepsilon}\|f\|_{L^p(\R^n)}, \;\quad \; {\rm supp}\; \widehat{f}\subset {\rm A} (1).
	\end{align}
	
	Indeed, by Corollary \ref{lemma-2}, we have
	\begin{align}
		\|e^{it\sqrt{-\Delta}}f\|_{L^p(\R^n\times [-R,R])}\leq C_\varepsilon R^{n(\frac{1}{2}-\frac{1}{p})+\varepsilon}\|f\|_{L^p(\R^n)},\;\;\;{\rm supp}\; \widehat{f}\subset {\rm A}(1).
	\end{align}
	After rescaling, we get
	\begin{align}\label{eq:08}
		\|e^{it\sqrt{-\Delta}}f\|_{L^p(\R^n\times [1,2])}\leq C_\varepsilon R^{n(\frac{1}{2}-\frac{1}{p})-\frac{1}{p}+\varepsilon}\|f\|_{L^p(\R^n)},\;\;\;{\rm supp}\;\widehat{f}\subset {\rm A}(R/2).
	\end{align}
	
	Now we perform the standard Littlewood--Paley decomposition on $f$. Let  $\varphi$ be a radial bump function supported
	on the ball $|\xi|\leq 2$ and equal to 1 on the ball $|\xi|\leq 1$. For $N\in 2^{\mathbb{Z}}$, we define the Littlewood--Paley projection operators by
	\begin{align*}
		&\widehat{P_{\leq N}f}(\xi) := \varphi(\xi/N)\widehat{f}(\xi),
		\\ &\widehat{P_{> N}f}(\xi) :=
		(1-\varphi(\xi/N))\widehat{f}(\xi),
		\\ &\widehat{P_{N}f}(\xi) :=
		(\varphi(\xi/N)-\varphi({2\xi}/{N}))\widehat{f}(\xi).
	\end{align*}
	
	Then we write,
	$$e^{it\sqrt{-\Delta}}f=e^{it\sqrt{-\Delta}}P_{\leq 1}f
	+\sum_{N>1} e^{it\sqrt{-\Delta}} P_N f.$$
	By the fixed-time estimate \eqref{eq-02} we have 
	\beq\label{eq-25}
	\big\|e^{it\sqrt{-\Delta}}P_{\leq 1}f\big\|_{L^p(\R^n\times[1,2])}
	\lesssim \|P_{\leq 1}f\|_{L^p_{s_{p}}(\R^n)}\lesssim  \|f\|_{L^p(\R^n)}.
	\eeq
	By the triangle inequality and \eqref{eq:08}, Theorem \ref{theoa2} is proved.
	
	Recall that $K=K_0 R^{\tilde{\delta}}\ll R^\varepsilon$, where $K_0,\tilde \delta$ will be chosen to satisfy the requirement of the forthcoming argument.
	To prove 
	\begin{align}
		\|e^{it\sqrt{-\Delta}}f\|_{L^p(B^{n}_{R}\times [-R,R])}\leq C_\varepsilon R^{n(\frac{1}{2}-\frac{1}{p})+\varepsilon}\|f\|_{L^p(\R^n)}, \;\quad \; {\rm supp}\; \widehat{f}\subset {\rm A} (1),
	\end{align}
	it suffices to show 
	\begin{align}\label{eq 216}
		\|e^{it\phi_R(D)}f\|_{L^p(B^{n}_{R}\times[-R,R])}\leq C_\varepsilon R^{n(\frac{1}{2}-\frac{1}{p})+\varepsilon}\|f\|_{L^p(\R^n)}, \; \; {\rm supp} \widehat{f}\subset  {\rm N}_{\varepsilon_0}(e_n),
	\end{align}
	where $\phi_R$ is in the class $\mathbf\Phi(R)$.
	
	Indeed, since ${\rm supp} \widehat{f}\subset {\rm A}(1)$, we decompose ${\rm A}(1)$ into a collection of finitely-overlapping sectors $\tau$ of radius $K^{-3}$ in the angular direction. Write $f=\sum_\tau f^\tau$ where $f^\tau$ is Fourier supported in $\tau$. Then
	$$\|e^{it\sqrt{-\Delta}}f\|_{L^p(B^{n}_{R}\times[-R,R])}\leq \sum_\tau \|e^{it\sqrt{-\Delta}}f^\tau\|_{L^p(B^{n}_{R}\times [-R,R])}.$$
	Given $\tau$, let $A_\tau$ be an orthogonal matrix such that $A_\tau {\bf e}_n=\eta_\tau$.
	By changing of variables:
	\beqq
	\xi\longrightarrow A_\tau \xi,
	\eeqq
	we rotate the sector $\tau$ to a sector  centered at $e_n$. Correspondingly, we make another change of variables with respect to $x$, i.e.
	\beqq
	x\longrightarrow  A_\tau x.
	\eeqq
	Then after sending
	\beqq
	x\longrightarrow  x-te_n,
	\eeqq
	correspondingly, the phase function becomes
	\beqq
	x\cdot \xi+t(|\xi|-\xi_n).
	\eeqq
	Finally we perform another change of variables with respect to $x,t,\xi$ as follows
	\beqq
	\xi'\longrightarrow  K^{-3}\xi', \,x'\longrightarrow  K^3x',\,t\longrightarrow   K^6t.
	\eeqq
	We claim that after the above change of variables, the resulting phase function is now in the class $\mathbf\Phi(R)$. Indeed, $\phi_R(\xi)$ is given by
	\beqq
	\phi_R(\xi)=K^6\big(\sqrt{K^{-6}|\xi'|^2+\xi_n^2}-\xi_n\big).\eeqq
	By the homogeneity of the phase function, we have
	\beqq
	\phi_R(\xi)=K^6\xi_n\Big(\big(1+\Big|\tfrac{K^{-3}\xi'}{\xi_n}\Big|^2\big)^{\f{1}{2}}-1\Big)=\frac{\xi_1^2+\ldots+\xi_{n-1}^2}{2\xi_n}+K^{-4}{\rm E}_R(\xi),
	\eeqq
	where 
	\beqq
	{\rm E}_R(\xi)=-\frac{K^{-2}}{2}\frac{|\xi'|^4}{\xi_n^3}\int_0^1(1-s) \Big(1+s K^{-6}\big|\tfrac{\xi'}{\xi_n}\big|^2\Big)^{-\frac{3}{2}}ds .
	\eeqq
	
	For fixed $N_{\rm par}\in\mathbb N,$ by choosing $K_0$ sufficiently large, we have  
	\beqq
	|\partial^\alpha {\rm E}_R(\xi)|<c_{\rm par}, \; 0 \leq |\alpha|\leq N_{{\rm par}}.
	\eeqq
	Thus, it suffices to estimate
	\beqq
	e^{it\phi_R(D)}g:=\int_{\R^{n}} e^{i(x'\cdot\xi'+x_n\xi_n+t\phi_R(\xi))}\widetilde a(\xi',\xi_n)\widehat{g}(\xi', \xi_n) d \xi,
	\eeqq
	where
	\begin{equation*}
		\widehat{g}(\xi)=\widehat{f}^\tau\Big(A_\tau(K^{-3}\xi',\xi_n)\Big).
	\end{equation*}
	A direct calculation shows that ${\rm supp}\;\widehat{g}\subset  {\rm N}_{\varepsilon_0}(e_n)$. We have finished verifying the claim.
	
	Combining the above estimates and \eqref{eq 216},  we have
	\begin{align*}
		\|e^{it\sqrt{-\Delta}}f^\tau\|_{L^p(B_R^{n}\times [-R,R])}&=K^{O(1)}\big\|e^{it\phi_R(D)}g\|_{L^p(B_R^{n}\times [-R,R])}\\
		\lesssim_\varepsilon & R^{n(\f{1}{2}-\f{1}{p})+\varepsilon}\|f\|_{L^p(\R^n)}. \end{align*}
	
	Let $Q_p(R)$ denote the smallest constant such that the following inequality holds for all phase functions $\phi_R$ in the class $\mathbf\Phi(R)$,
	\begin{equation}
		\|e^{it\phi_R(D)}f\|_{L^p(B_R^{n}\times [-R,R])}\leq Q_p(R)R^{n(\f{1}{2}-\f{1}{p})}\|f\|_{L^p},\quad {\rm supp}\; \widehat{f}\subset {\rm N}_{\varepsilon_0}(e_n).
	\end{equation}
	To prove \eqref{eq:301}, it suffices to show 
	$$Q_p(R)\lesssim_\varepsilon R^\varepsilon.$$
	
	
	\begin{remark}
		We want to emphasize that reducing to the class $\mathbf\Phi(R)$ is only needed for proving the narrow decoupling estimate in Section \ref{section-7}. The statements in Section \ref{wavee} through \ref{section-5}, including our $k$-broad ``norm" estimates, hold true for all general phase functions $\phi$ in the class $\mathbf \Phi$. Moreover, in the above reductions, we start with the standard circular cone given by the phase function $|\xi|$, while it can be easily seen that a similar argument works for any phase function satisfying conditions ${\bf H_1}$ and ${\bf H_2}$.  Indeed, one can see from the following formula
		\beqq
		\phi(\xi',\xi_n)-\nabla\phi(e_n)\cdot (\xi',\xi_n)=\f{\langle \partial^2_{\xi'\xi'}\phi(e_n)\xi',\xi'\rangle }{2\xi_n} +\sum_{|\alpha|=3}\frac{3}{\alpha!}\int_0^1 (1-s)^2(\partial^\alpha \phi)(s\tfrac{\xi'}{\xi_n},1)ds\tfrac{(\xi')^\alpha}{\xi_n^2}.
		\eeqq
		Thus our local smoothing bounds in Theorem \ref{theoa2} also hold true for such phase functions.
	\end{remark}
	
	\section{Wave packet decomposition}\label{wavee}
	\subsection{Construction of wave packet} In this section, we present the wave packet decomposition and collect some useful properties that we shall need from \cite{OW}.  In this section and the next, we shall consider a phase function $\phi$ in the bigger class $\mathbf \Phi$. Same arguments would work for any $\phi_R$ in the smaller class $\mathbf \Phi(R)$.

	Fix a large constant $R\gg 1$. We cover the annulus ${\rm A}(1)$ using a collection of $1\times R^{-1/2}\cdots\times R^{-1/2}$ sectors $\nu$ with finite overlaps. Let $\{\psi_\nu\}$ be a smooth partition of unity subordinate to this cover, and write $f=\sum_{\nu} f_\nu$, where $f_\nu:=\psi_\nu f$. 
	
	Next, we further decompose $f_\nu$ on the physical side. Cover $\R^n$ by a collection  of finitely overlapping balls $B_w:=B^n_{R^{{(1+\delta)}/{2}}}(w)$ of radius $R^{\f{1+\delta}{2}}$ centered at $w\in R^{\f{1+\delta}{2}}\mathbb{Z}^n$, where $\delta>0$ is a fixed small constant. Let $\eta_w$ be a smooth partition of unity subordinate to this cover, write $f=\sum_{\nu,w}f_{\nu,w}$, where $f_{\nu,w}:=\big(\eta_w(\psi_\nu f)^{\wedge}\big)^{\vee}.$  For given $\nu,w$, we further decompose the ball  $B_w$  into $R^{(1+\delta)/2}$ plates $\{P_{\nu,w}^\ell\}_{\ell}$  of dimension $1$ in the direction parallel to $\partial_\eta\phi(\xi_\nu)$ and $R^{\f{1+\delta}{2}}$ in all the other directions, 
	where $\xi_\nu \in S^{n-1}$ denotes the direction of the center-line of the sector $\nu$. Let $\eta_{\nu,w}^\ell$ be a smooth partition of unity subordinate to this cover.  We write $f=\sum_{\nu,w,\ell}\big(\eta_{\nu,w}^\ell\eta_w(\psi_\nu f)^{\wedge}\big)^{\vee}.$
	
	Finally, Let $\tilde{\psi}_\nu$ be a smooth function which is essentially support on $\nu$, and $\widetilde \psi_\nu=1$ on the $cR^{-1/2}$-neighborhood of the support of $\psi_\nu$ where $c>0$ is a small constant. Define
	\beqq
	f_{\nu,w}^\ell:=\widetilde \psi_\nu\big(\eta_{\nu,w}^\ell \eta_w(\psi_\nu f)^{\wedge}\big)^{\vee},
	\eeqq
	then it is straight forward to check that
	$$\|f_{\nu,w}^\ell-\big(\eta_{\nu,w}^\ell \eta_w(\psi_\nu f)^{\wedge}\big)^{\vee}\|_{L^\infty}\leq {\rm RapDec}(R)\|f\|_{L^2}.$$
	
	We may decompose $f$ as follows
	\beqq
	f=\sum_{\nu,w,\ell} f_{\nu,w}^\ell+{\rm RapDec}(R)\|f\|_{L^2}.
	\eeqq
	The functions $\{f_{\nu,w}^\ell\}$ are orthogonal in the sense that: for a set $\mathcal{T}$ of triplets $(\nu,w,\ell)$, we have
	\beqq
	\left\|\sum_{(\nu,w,\ell)\in\mathcal{T}}f_{\nu,w}^\ell\right\|_{L^2}^2\sim \sum_{(\nu,w,\ell)\in\mathcal{T}}\Big\|f_{\nu,w}^\ell\Big\|_{L^2}^2.
	\eeqq
	
	Now we define the associated tube $T_{\nu,w}^\ell$ by 
	\beq\label{wavepa}
	\begin{aligned}
		T_{\nu,w}^\ell:=\big\{(x,t)\in \R^{n+1},|t|\leq R: |\Pi_\nu(x-w_\ell+t\partial_\xi&\phi(\xi_\nu))|\leq CR^\delta,\\ &|\Pi_{\nu^\perp}(x-w_\ell+t\partial_\xi\phi(\xi_\nu))|\leq CR^{\f{1}{2}+\delta}\big\}.
	\end{aligned}
	\eeq
	where $\Pi_\nu,\Pi_{\nu^\perp}$ denote the orthogonal projection operator defined by
	$$\Pi_\nu(\xi):=(\xi\cdot \xi_\nu) \xi_\nu, \; \Pi_{\nu^{\perp}}(\xi):=\xi-\Pi_\nu(\xi).$$
	and $w_\ell\in \R^n$ is the center of the plate $P_{\nu,w}^\ell$. Define $${\bf L}(\nu):=\frac{1}{\sqrt{1+|\nabla\phi(\xi_\nu)|^2}}(-\nabla\phi(\xi_\nu),1).$$ From \eqref{wavepa}, one can see that $T_{\nu,w}^\ell$ 
	intersects the hyperplane $t=0$ at $P_{\nu,w}^\ell$ and satisfies 
	$$T_{\nu,w}^\ell \subset N_{CR^\delta}(P_{\nu,w}^\ell+ t{\bf L}(\nu)), \;\; |t|\leq CR.$$
	
	We define the extension operator associated to the general cone $(\xi, \phi(\xi))$ by 
	\beqq
	Ef(x,t):=\int_{{\rm A}(1)}e^{i(x\cdot \xi+t\phi(\xi))}f(\xi)d \xi.
	\eeqq
	
	The following lemma shows that each wave packect $Ef_{\nu,w}^\ell$ is essentially localized to the tube $T_{\nu,w}^\ell$ in physical space.
	\begin{lemma}\label{lew}
		If $(x,t)\in B_R^{n+1}(0)\backslash T_{\nu,w}^\ell$, then
		\beqq
		|Ef_{\nu,w}^\ell(x,t)|\leq {\rm RapDec}(R)\|f\|_{L^2}.
		\eeqq
	\end{lemma}
	
	The proof of Lemma \ref{lew} is standard, and for instance, can be obtained by slightly modifying the proof of Lemma 2.1 in \cite{OW}.
	
	\subsection{Comparing wave packet at different scales}
	
	Suppose $B_\rho^{n+1}(y)\subset B_R^{n+1}(0)$ for some radius $R^{1/2}<\rho<R$. We need to decompose $f$ into wave packets over the ball $B_\rho^{n+1}(y)$ at this smaller spatial scale $\rho$. 
	
	We apply a transformation $z=y+\tilde z$ to recenter $B_\rho^{n+1}(y)$, here $z:=(x,t), \tilde{z}:=(\tilde x,\tilde t)$. Define
	\beqq
	\phi_y(\xi):=y'\cdot \xi+y_{n+1}\phi(\xi),\quad y=(y',y_{n+1})
	\eeqq
	then,
	$$Ef(z)=E\tilde f(\tilde z),$$
	where $\tilde f(\eta)=e^{i\phi_y(\eta)}f(\eta).$
	We now perform wave packet decomposition with respect to $\tilde f$ at scale $\rho$. Following the construction in the last section, we write
	\beqq
	\tilde f=\sum_{\tn,\tw,\tl} \tilde{f}_{\tn,\tw}^{\tilde{\ell}},
	\eeqq
	where $\tn \subset {\rm A}(1)$ is a sector of width about $\rho^{-1/2}$ in the angular direction and length about 1 in the radial direction, $\tilde{w}\in \rho^{\f{1+\delta}{2}}\mathbb{Z}^n$. The Fourier support of $\tilde{f}_{\tn,\tw}^{\tl}$ is essentially contained in  a thin plate  $P_{\tn,\tw}^{\tl}$ of side length $\rho^{1/2+\delta}$ and thickness $\rho^{\delta}$ in the ball of radius $\rho^{1/2+\delta}$ centered at $\tw$. $E\tilde f_{\tn,\tw}^{\tl}$ is essentially supported in the tube $\widetilde{T}_{\tn, \tw}^{\tl}$ with 
	\beqq
	P_{\tn,\tw}^{\tilde{\ell}}+\f{R}{C}{\bf L}(\tn)\subset \widetilde{T}_{\tn,\tw}^{\tilde{\ell}}\subset P_{\tn,\tw}^{\tilde{\ell}}+CR{\bf L}(\tn),\;\; C>0 \;\text{sufficiently large}.
	\eeqq
	
	A natural question then appears: how this new wave packet decomposition at a smaller scale $\rho$, $\tilde f=\sum_{\tn,\tw,\tl}\tilde f_{\tn,\tw}^{\tl}$,  relates to the original wave packet decomposition $f=\sum_{\nu,w,\ell}f_{\nu,w}^\ell$ at scale $R$? To be more precise, for a given $(\nu,w,\ell)$,  which $(\tn,\tw,\tl)$ contributes significantly to the wave packet $f_{\nu,w}^\ell$? To answer this question, we first define
	$$\widetilde{\mathbb{T}}_{\nu,w,\ell}:=\{(\tn,\tw,\tl): {\rm Ang}(\nu, \tn)\lesssim \rho^{-1/2}, {\rm Dist}(P_{\tn, \tw}^{\tl}, P_{\nu,w}^\ell+P_{\nu}-\partial_\eta\phi_y(\xi_\nu))\lesssim R^\delta\}.$$
	where $P_\nu$ is given by
	$$P_\nu:=\{x \in \R^{n}: |\Pi_\nu(x)|\leq CR^\delta, |\Pi_{\nu^\perp}(x)|\leq CR^{\f{1+\delta}{2}}\}.$$
	The following lemma shows the relationship between wave packet decomposition at different scales.
	\begin{lemma}\label{le:le}
		$\big(f_{\nu,w}^\ell\big)^{\widetilde{}}$ is essentially made of small wave packets from $\widetilde{\mathbb{T}}_{\nu,w,\ell}$. In other words,  \beqq
		\big(f_{\nu,w}^\ell\big)^{\widetilde{}}=\sum_{(\tn,\tw,\tl)\in \widetilde{\mathbb{T}}_{\nu,w,\ell}}\Big(\big(f_{\nu,w}^\ell\big)^{\widetilde{}}\Big)_{\tn,\tw}^{\tl}+{\rm RapDec}(R)\|f\|_{L^2}.
		\eeqq
	\end{lemma}

	Next, we explore a geometric features of a tube $T_{\tn,\tw}^{\tl}$ with $(\tn,\tw,\tl)\in \widetilde{\mathbb{T}}_{\nu,w,\ell}$.
	
	\begin{lemma}\label{tanc}
		For any $(\tn,\tw,\tl)\in \widetilde{\mathbb{T}}_{\nu,w,\ell}$, there holds
		$${\rm Ang}(\nu, \tn)\lesssim \rho^{-1/2},$$
		and
		$${\rm Dist}([T_{\nu,w}^\ell\cap B_\rho^{n+1}(y)]+2P_\nu, T_{\tn,\tw}^{\tl})\lesssim R^\delta.$$
	\end{lemma}
	The proof of Lemma \ref{le:le} and Lemma \ref{tanc} are similar to that of Lemma 5.3 and Lemma 5.4 in \cite{OW}. We omit the proof here.
	
	Next, we will group large and small wave packets into different sub-collections. Let $\tn_0$ be a sector in ${\rm A}(1)$ of dimensions $\rho^{-1/2}$ in the angular direction and $\sim 1$ in the radial direction, and $w_0 \in R^{\f{1+\delta}{2}}\mathbb{Z}\cap B(0,\rho)$. We define the set $\widetilde{\mathbb{T}}_{\tn_0,w_0}, \mathbb{T}_{\tn_0,w_0}$ respectively  as follows:
	$$\widetilde{\mathbb{T}}_{\tn_0,w_0}:=\{(\tn,\tw,\tl): {\rm Ang}(\tn,\tn_0)\lesssim \rho^{-1/2},\quad P_{\tn,\tw}^{\tl}\subset B(w_0,R^{1/2+2\delta})\}.$$
	and
	$$\mathbb{T}_{\tn_0,w_0}:=\big\{(\nu,w,\ell): {\rm Ang}(\nu, \tn_0)\lesssim \rho^{-1/2},\;T_{\nu,w}^\ell \cap B_\rho^{n+1}(y)\subset B(w_0,R^{1/2+2\delta})+\rho {\bf L}(\tn_0)+\{y\}\big\}.$$
	
	For any given $\tn_0,w_0$ and function $g$,  we define $\widetilde{g}_{\tn_0,w_0}$ and $g_{\tn_0,w_0}$ respectively as follows:
	\beqq
	\widetilde{g}_{\tn_0,w_0}:=\sum_{(\tn,\tw,\tl)\in \widetilde{\mathbb{T}}_{\tn_0,w_0}}\tilde{g}_{\tn,\tw}^{\tl}, \; g_{\tn_0,w_0}:=\sum_{(\nu,w,\ell)\in \mathbb{T}_{\tn_0,w_0} }g_{\nu,w}^\ell.
	\eeqq
	Correspondingly, we have the wave packets decomposition for $g$ and $\tilde g$ in the sense that
	\beqq
	g=\sum_{(\tn_0,w_0)}g_{\tn_0,w_0}+{\rm RapDec}(R)\|f\|_{L^2},\quad \tilde{g}:=\sum_{(\tn_0,w_0)}\tilde{g}_{\tn_0,w_0}+{\rm RapDec}(R)\|f\|_{L^2}.
	\eeqq
	Furthermore, we have the following $L^2$-orthogonality property.
	\beqq
	\|g\|_{L^2}^2\sim \sum_{(\tn_0,w_0)}\|g_{\tn_0,w_0}\|_{L^2}^2,\quad \|\tilde{g}\|_{L^2}^2\sim \sum_{(\tn_0,w_0)}\|\tilde{g}_{\tn_0,w_0}\|_{L^2}^2.
	\eeqq

	With the definition above, for any $(\tn_0,w_0)$,  these two collections $\widetilde{\mathbb{T}}_{\tn_0,w_0}, \mathbb{T}_{\tn_0,w_0}$ are related in the sense that
	\beqq
	\widetilde{\mathbb{T}}_{\tn_0,w_0}=\bigcup_{(\nu,w,\ell) \in \mathbb{T}_{\tn_0,w_0}}\widetilde{\mathbb{T}}_{\nu,w,\ell}.
	\eeqq
	Finally, we have
	\begin{lemma}
		If $g$ is concentrated on large wave packets in $\mathbb{T}_{\tn_0,w_0}$, then $\tilde g$ is concentrated on small wave packets in $\widetilde{\mathbb T}_{\tn_0,w_0}$. On the other hand, if $\tilde g$ is concentrated on small wave packets in $\widetilde{\mathbb{T}}_{\tn_0,w_0}$, then $g$ is concentrated on large wave packets on $\mathbb{T}_{\tn_0,w_0}$.
	\end{lemma}
	
	\section{A Geometric lemma}\label{section-4}
	In this section, we establish a geometric lemma associated with the cone given by $(\xi,\phi(\xi))$, for any given phase functions $\phi$ in the class $\mathbf \Phi.$
	First, let us give a slightly different version of Lemma 5.8 in \cite{OW} for the model case $\phi(\xi)=|\xi|$, which will shed light on the case of general cones.
	\subsection{ A geometric lemma for the circular cone $(\xi,|\xi|)$} For convenience, we will use $\mathcal{C}$ to denote the truncated cone, i.e. 
	$\mathcal{C}:=\{(\xi,|\xi|): 1/2\leq |\xi|\leq 1\}.$
	Let $V\subset \mathbb{R}^{n+1}$ be an $m$ dimensional affine subspace given by
	\beqq
	V:=\big\{(x_1,\cdots, x_{n+1})\in \R^{n+1}:\sum_{j=1}^{n+1}a_{i,j}x_j=b_i,\quad i=1,\cdots,n+1-m\big\}.
	\eeqq
	For a given point $(\xi, |\xi|)\in \mathcal{C}$, the unit normal vector of $\mathcal{C}$ at $(\xi, |\xi|)$ is $\mathbf{n}_\xi=\f{1}{\sqrt{2}}\big(-\f{\xi}{|\xi|},1\big)$. Therefore, all the points on the cone $\mathcal{C}$  of which the normal vectors are parallel to $V$ lie in an affine subspace $\bar V$ defined by
	\beqq
	\bar V:=\{(x_1,\cdots, x_{n+1})\in \R^{n+1}: \sum_{j=1}^{n}a_{i,j}x_j-a_{i,n+1}x_{n+1}=0,\; i=1,\cdots,n+1-m\}.
	\eeqq
	Assume that the unit normal ${\bf n}_\xi$ is parallel to $V$, i.e.
	\beqq
	\sum_{j=1}^na_{i,j}\f{\xi_j}{|\xi|}-a_{i,n+1}=0,\;i=1,\cdots,n+1-m,
	\eeqq
	which implies that
	\begin{equation}\label{maxrank}
		{\rm rank}\left( \begin{array}{ccc}
			a_{1,1} & \cdots &a_{1,n}\\
			\vdots & \ddots & \vdots\\
			a_{n+1-m,1}&\cdots& a_{n+1-m,n}
		\end{array} \right)=n+1-m.
	\end{equation}
	We denote by $V^{-}\subset \R^n$ a subspace defined as
	\beqq
	V^{-}:=\big\{(x_1,\cdots, x_n)\in \R^n: \sum_{j=1}^{n}a_{i,j}x_j=0,\; i=1,\cdots,n+1-m\big\}.
	\eeqq
	From \eqref{maxrank}, we see that $V^{-}$ is an $(m-1)$-dimensional subspace.
	
	For any $m$-dimensional linear subspace $\bar V$, if $\bar V\cap \mathcal{C}\neq \varnothing$, then $\bar V$ intersects the light cone either tangentially or transversally, which is demonstrated in Figure \ref{fig2} 
	\begin{figure}
		\centering
		\begin{subfigure}{.5\textwidth}
			\centering
			\includegraphics[width=1\linewidth]{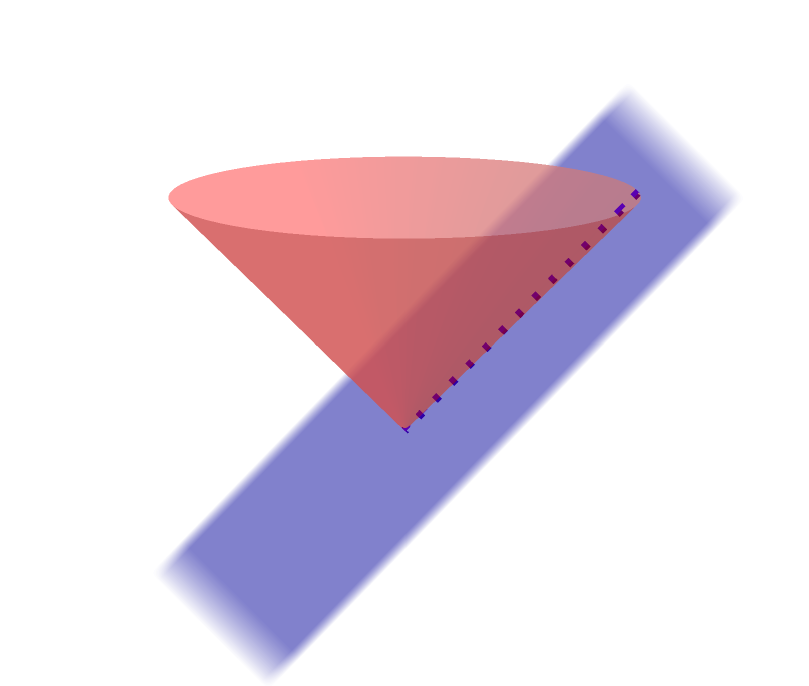}
			\caption*{Tangential intersection}
			\label{fig:sub1}
		\end{subfigure}%
		\begin{subfigure}{.5\textwidth}
			\centering
			\includegraphics[width=1\linewidth]{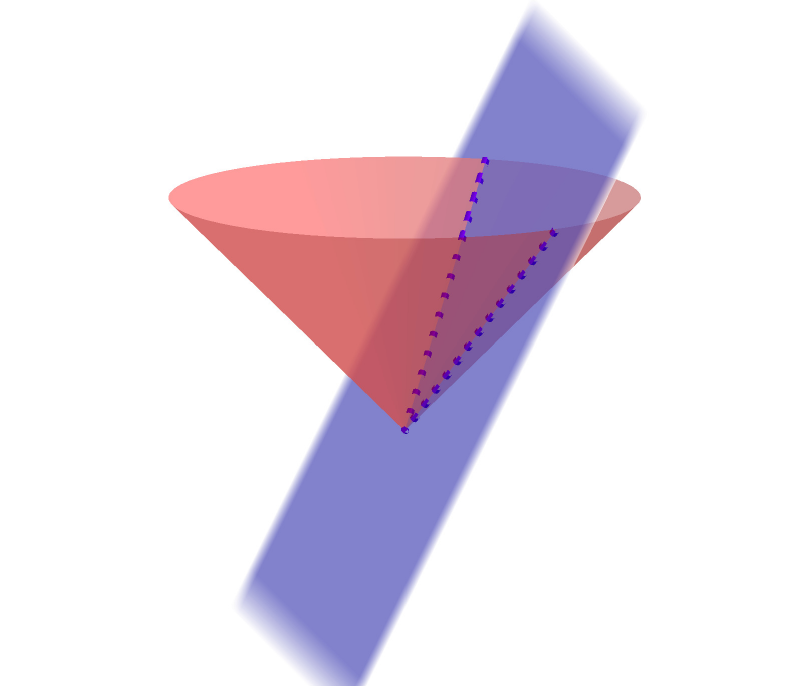}
			\caption*{Transversal intersection}
			\label{fig:sub2}
		\end{subfigure}
		\caption{Tangential and transversal intersections.}
		\label{fig2}
	\end{figure}. Based on the above observation, we start with two geometric lemmas concerning the circular cone, which corresponds to Lemma 5.8 in \cite{OW}.
	\begin{lemma}\label{le-ge}
		Assume $\eta\in S^{n-1}$. If $(\eta,1)\in \bar V$ and ${\rm Ang}(\eta,V^{-})>\f{\pi}{2}-K^{-2}$, then $\bar V\cap \mathcal{C}\subset \{t(\xi,1): t\in \R, \xi\in S^{n-1}, {\rm Ang}(\xi,\eta)\lesssim K^{-2}\}$.
	\end{lemma}
	\begin{proof}
		Let $\alpha_i=(a_{i,1},\cdots, a_{i,n})$. Since ${\rm Ang}(\eta,V^{-})>\f{\pi}{2}-K^{-2}$, there exists a vector $\bar \eta\in S^{n-1}$ such that
		\beq\label{eq:li}
		\bar \eta=(\bar \eta_1,\cdots, \bar \eta_{n})=\sum_{i=1}^{n+1-m}\lambda_i\alpha_i, \lambda_i\in \R, \; {\rm Ang}(\eta,\bar \eta)\leq K^{-2}.
		\eeq
		Note that $(\eta,1)\in \bar V$, we have
		\beqq
		\sum_{j=1}^{n}\Big(\sum_{i=1}^{n+1-m}\lambda_ia_{i,j}\Big)\eta_j-\Big(\sum_{i=1}^{n+1-m} \lambda_ia_{i,n+1}\Big)=0.
		\eeqq
		By \eqref{eq:li}, we obtain
		\beqq
		\sum_{j=1}^{n}\bar \eta_j\cdot \eta_j-\Big(\sum_{i=1}^{n+1-m}\lambda_i a_{i,n+1}\Big)=0,
		\eeqq
		Note that ${\rm Ang}(\eta,\bar \eta)\leq K^{-2}$, thus $\bar \eta\cdot \eta>1-K^{-4}.$ It follows that $$\bar \eta_{n+1}:=\sum_{i=1}^{n+1-m}\lambda_i a_{i,n+1}>1-K^{-4}.$$ For all $(\xi, |\xi|) \in \mathcal{C}\cap \bar{V}$, since $(\bar \eta, -\bar \eta_{n+1})\in (\bar V)^{\bot}$, we have 
		$$\xi\cdot \bar \eta-|\xi|\bar{\eta}_{n+1}=0.$$
		Note that $|\bar \eta|=1$ and $\bar \eta_{n+1}>1-K^{-4}$,  we obtain  ${\rm Ang}(\xi,\bar \eta)\lesssim K^{-2}$.
	\end{proof}
	Decompose $\R^{n+1}=\bar V\bigoplus W$, that is,  $W$ is the orthogonal complement subspace of $\bar V$ in $\R^{n+1}$.
	\begin{lemma}\label{eq-le2}
		If for each $(\eta,|\eta|)\in \mathcal{C}\cap \bar V$, ${\rm Ang}(\eta,V^{-})\leq \f{\pi}{2}-K^{-2}$, then $W$ and $V$ are transversal in the sense that ${\rm Ang}(V,W)\gtrsim K^{-4}$;
	\end{lemma}
	\begin{proof}
		Let $\{\beta_1,\cdots, \beta_{m-1}\}$ be an orthogonal basis for $V^{-}$, we may choose another unit vector $\beta_m$ such that $\{\beta_1,\cdots,\beta_{m-1}, \beta_m\}$ forms a unit orthogonal basis for $V$. Assume $\eta\in S^{n-1}$, since $\bar \beta_m:=\frac{1}{\sqrt{2}}(\eta,1)\in \bar V$, it is easy to verify that $\{\beta_1, \cdots, \beta_{m-1}, \bar \beta_m \}$ forms a basis for the subspace $\bar V$. To prove $${\rm Ang}(V,W)\gtrsim K^{-4},$$
		it suffices to show for any unit vector $v$ parallel to $V$, $|{\rm Proj}_{\bar V}v|\gtrsim K^{-4}.$  To achieve this, we subdivide the vectors parallel $V$ into two categories.
		
		{\bf \noindent Case I: \;$|{\rm Proj}_{V^{-}}v|>\f{1}{2}$.} Since that $V^{-}\subset\bar V$, we have $$|{\rm Proj}_{\bar V}v|> \frac12.$$
		
		{\bf \noindent Case II:\; $|{\rm Proj}_{V^{-}}v|\leq \f{1}{2}$.} In this case, we have $|{\rm Proj}_{\beta_m}v|> \f{1}{2}$. To prove $|{\rm Proj}_{\bar V}v|\gtrsim K^{-4}$, it suffices to show
		$$|{\rm Proj}_{\bar{\beta}_m}\beta_m|\gtrsim K^{-4},$$
		since
		$$ |{\rm Proj}_{\bar V}v|\gtrsim |{\rm Proj}_{\bar{\beta}_m}\beta_m|\gtrsim K^{-4}.$$
		By our construction $\beta_{m}\bot V^{-}$,  there exists $(\lambda_1,\cdots,\lambda_{n+1-m})\in\mathbb R^{n+1-m}$ such that
		$$(\sum_{i=1}^{n+1-m} \lambda_i a_{i,1},\cdots,\sum_{i=1}^{n+1-m} \lambda_i a_{i,n},c_{n+1})= \beta_{m}.$$
		For convenience, denote $c_j=\sum\limits_{i=1}^{n+1-m} \lambda_i a_{i,j}$ with $\sum\limits_{j=1}^{n+1}c_j^2=1$.
		
		Take $\tilde c_{n+1}\in\mathbb R$ such that $(c_1,\cdots, \tilde{c}_{n+1})\in V^{\perp}$. Since $\big(-\frac{\xi}{|\xi|},1\big)\parallel V$, it follows that
		\beq
		c_1\f{\xi_1}{|\xi|}+\cdots+c_n\f{\xi_n}{|\xi|}-\tilde{c}_{n+1}=0.
		\eeq
		By Cauchy-Schwarz inequality, we have
		\beq\label{eqccc}
		\tilde{c}_{n+1}^2\leq \sum_{j=1}^n c_j^2.
		\eeq
		Since $(c_1,\cdots, c_{n+1})\in V,\  (c_1,\cdots, \tilde{c}_{n+1})\in V^{\perp}$, we have
		\beqq
		\sum_{j=1}^n c_j^2+c_{n+1}\tilde{c}_{n+1}=0.
		\eeqq
		Together with \eqref{eqccc}, which implies that
		\beq
		\sum_{j=1}^nc_j^2\leq c_{n+1}^2.
		\eeq
		Therefore, we have \beqq c_{n+1}\geq \f{\sqrt{2}}{2}, \;\;\big(\sum_{j=1}^n c_j^2\big)^{1/2}\leq \f{\sqrt{2}}{2}.\eeqq
		The magnitude of the projection of $\beta_m$ onto $\bar \beta_m$ equals
		$$\f{1}{\sqrt{2}}\Big|c_1\eta_1+\cdots+c_n\eta_n+c_{n+1}\Big|.$$
		Recall that ${\rm Ang}(\eta, V^{-})<\f{\pi}{2}-K^{-2}$, it follows that
		$$|c_1\eta_1+\cdots+c_n\eta_n|<\f{\sqrt{2}}{2}(1-CK^{-4}).$$
		Since $c_{n+1}\geq \f{\sqrt{2}}{2}$, we obtain
		$$\frac12\Big|c_1\eta_1+\cdots+c_n\eta_n+c_{n+1}\Big|> CK^{-4}.$$
	\end{proof}

	We summarize the above discussion as follows:
	\begin{lemma}\label{eqaaaa}
		Decompose $\R^{n+1}=\bar V\bigoplus W$ such  that $\bar V\bot W$. We have the following dichotomy.
		\begin{itemize}
			\item[{\bf\noindent $\bullet$ }] If for each $(\eta,|\eta|)\in \mathcal{C}\cap \bar V$, ${\rm Ang}(\eta,V^{-})\leq \f{\pi}{2}-K^{-2}$, then $W$ and $V$ are transversal in the sense that ${\rm Ang}(V,W)\gtrsim K^{-4}$;
			\item[{\bf \noindent $\bullet$  }] If there exists $(\eta,|\eta|)\in \mathcal{C}\cap \bar V$ such that ${\rm Ang}(\eta,V^{-})>\f{\pi}{2}-K^{-2}$, then the projection of $\mathcal{C}\cap \bar V$ onto $\R^n$ is contained in a slab of dimensions $\sim 1\times K^{-2}\times \cdots \times  K^{-2}.$
		\end{itemize}
	\end{lemma}

	\subsection{ Generalization of Lemma \ref{eqaaaa} to a general class of cones }\label{gem}
	
	Now, we generalize the above lemma to the class of cones which are in the class $\mathbf\Phi$ as defined in Section 2.

	We define a set $L$ by 
	\beqq
	L:=\{\xi\in {\rm A}(1): \sum_{j=1}^{n} a_{i,j}\partial_{\xi_j}\phi(\xi)-a_{i,n+1}=0;\; i=1,\cdots,n+1-m\}.
	\eeqq
	If we choose $\phi(\xi)=|\xi|$, then  $\{(\xi,|\xi|): \xi \in L\}$ lies in the subspace $\bar V$. Therefore, if $L$ is not empty, the dimension of $L$ depends on whether $\bar V$ intersects the light cone $(\xi, |\xi|)$ tangentially or transversally. To be more precise, if $L$ intersects the light cone tangentially, then $\dim L=1$, otherwise $\dim L=m-1$. In this section, we shall prove that this fact can be generalized to general cones satisfying the homogeneous convex conditions.
	
	\begin{lemma}
		If $L$ is not empty, then $\dim  L=1$ or $\dim  L=m-1$.
	\end{lemma}
	\begin{proof}
		Denote $\alpha_i:=(a_{i,1},\cdots, a_{i,n})$ and $F=(F_1,\cdots,F_{n+1-m})$ where
		\beqq
		F_i:=\sum_{j=1}^na_{i,j}\partial_{\xi_j}\phi(\xi)-a_{i,n+1}.
		\eeqq
		For convenience, we will use ${\rm Hess}(\phi)$ to denote the Hessian matrix of $\phi$, i.e.
		\beqq
		{\rm Hess}(\phi)=\Big(\frac{\partial^2\phi}{\partial \xi_i\partial \xi_j}\Big)_{n\times n}.
		\eeqq
		By the homogeneous convex conditions, we see that the $0$-eigenspace of ${\rm Hess}(\phi)$ at $\eta$ is spanned by the vector $\eta$, i.e.
		\beq\label{orth}
		{\rm Hess}(\phi)\big|_{\xi=\eta} \eta=0\eta=0.
		\eeq
		If $\eta\in L \;\text{and}\; \eta\not\in {\rm span}\,\{\alpha_1,\cdots,\alpha_{n+1-m}\}$, then
		\begin{displaymath}
			{\rm rank}\left({\rm Hess}(\phi)\Big|_{\xi=\eta} \left( \begin{array}{ccc}
				a_{1,1} & \cdots &a_{1,n}\\
				\vdots & \ddots & \vdots\\
				a_{n+1-m,1}&\cdots& a_{n+1-m,n}
			\end{array} \right)\right)=n+1-m.
		\end{displaymath}
		By the implicit function theorem, we have
		\beqq
		\dim (L)=m-1.
		\eeqq
		If $\eta\in L \; \text{and}\; \eta\in {\rm span}\,\{\alpha_1,\cdots,\alpha_{n+1-m}\}$, by the homogeneity  of $\phi$, we may assume that $\eta\in S^{n-1}$.  If there is another vector $\bar \eta\in S^{n-1}\cap L$ not parallel to $\eta$, a simple calculation using the definition of $L$ gives that
		\beq\label{para}
		\big(\partial_\xi\phi(\eta)-\partial_{\xi}\phi(\bar \eta)\big)\cdot \eta=0.
		\eeq
		Using Taylor's expansion formula, we obtain
		\beq\label{taylor}
		\big(\partial_\xi\phi(\eta)-\partial_{\xi}\phi(\bar \eta)\big)\cdot \eta=\langle \bar \eta-\eta, {\rm Hess}(\phi)\Big|_{\xi=\eta}\eta\rangle+\f{1}{2} \partial_{\xi\xi}^2\langle\partial_\xi\phi, \eta\rangle\Big|_{\xi=\eta}(\bar \eta-\eta)^2+O(\bar \eta-\eta)^3.
		\eeq
		By \eqref{orth}, we have
		\beqq
		\langle \bar \eta-\eta, {\rm Hess}(\phi)\Big|_{\xi=\eta}\eta\rangle=0.
		\eeqq
		By the homogeneity of $\phi$, it follows
		$$\partial_{\xi\xi}^2\langle\partial_\xi\phi, \eta\rangle\Big|_{\xi=\eta}=- \Big(\frac{\partial^2\phi}{\partial \xi_i\partial \xi_j}\Big)\Big|_{\xi=\eta},$$
		since $\bar \eta$ is not parallel to $\eta$, using the homogeneity convex conditions of $\phi$, it follows that
		$$\Big|\f{1}{2} \partial_{\xi\xi}^2\langle\partial_\xi\phi, \eta\rangle\Big|_{\xi=\eta}(\bar \eta-\eta)^2\Big|\sim |\bar\eta-\eta|^2,$$
		which contradicts \eqref{para}. Therefore, if $\eta\in L$ and $\eta\in {\rm span}\,\{\alpha_1,\cdots,\alpha_{n+1-m}\}$, then
		\beqq
		L\subset \{t\eta:  t\in \R\}.
		\eeqq

	\end{proof}
	
	Let $V, V^{-}$ be as defined in Section 4.1.  Next, we will generalize Lemma \ref{le-ge} and \ref{eq-le2} to general cones.
	
	{\bf \noindent Case I: Tangential case.}
	
	\begin{lemma}
		Let $\eta\in S^{n-1}$. If $\eta \in L$ and ${\rm Ang}(\eta,V^{-})>\f{\pi}{2}-K^{-2}$, then $L$ is contained in the set
		$$\{\xi\in \R^n: {\rm Ang}(\xi,\eta)\lesssim K^{-2}\}.$$
	\end{lemma}
	\begin{proof}
		Let $\eta\in L\cap S^{n-1}$ with ${\rm Ang}(\eta,V^{-})>\f{\pi}{2}-K^{-2}$. Consider another unit vector $\bar \eta \in L\cap S^{n-1}$, we need to show that
		\beqq
		{\rm Ang}(\eta,\bar \eta)\lesssim K^{-2}.
		\eeqq
		By the definition of $L$, we have
		\beqq
		\sum_{j=1}^n a_{i,j}(\partial_{\xi_j}\phi(\eta)-\partial_{\xi_j}\phi(\bar \eta))=0, \quad i=1,\cdots,n+1-m.
		\eeqq
		Since ${\rm Ang}(\eta,V^{-})>\f{\pi}{2}-K^{-2}$ and $(\partial_{\xi}\phi(\eta)-\partial_{\xi}\phi(\bar \eta))\in V^{-}$,  it follows that
		\beq
		\Big|(\partial_{\xi}\phi(\eta)-\partial_{\xi}\phi(\bar \eta))\cdot \eta\Big|\lesssim |\partial_\xi\phi( \eta)-\partial_{\xi}\phi(\bar \eta)| K^{-2}.
		\eeq
		As in \eqref{taylor},
		\beq
		\big(\partial_\xi\phi(\eta)-\partial_{\xi}\phi(\bar \eta)\big)\cdot \eta=\langle \bar \eta-\eta, {\rm Hess}(\phi)\Big|_{\xi=\eta}\eta\rangle+\f{1}{2} \partial_{\xi\xi}^2\langle\partial_\xi\phi. \eta\rangle\Big|_{\xi=\eta}(\bar \eta-\eta)^2+O(\bar \eta-\eta)^3.
		\eeq
		Again by using the homogeneous convex conditions we have $$|\eta-\bar \eta|\lesssim |\partial_\xi\phi( \eta)-\partial_{\xi}\phi(\bar \eta)|^{\f{1}{2}} K^{-1},$$
		while
		$$|\partial_\xi\phi( \eta)-\partial_{\xi}\phi(\bar \eta)|\sim |\bar \eta-\eta|,$$
		therefore, 
		$$|\bar \eta-\eta|\lesssim K^{-2}.$$
	\end{proof}
	
	{\bf \noindent Case II: Transversal case.} In general, $\{(\xi,\phi(\xi)): \xi \in L\}$ may not lie in an affine subspace, actually $L$ can be a curved submanifold. To generalize Lemma \ref{eq-le2}, we should construct the associated affine subspace $\bar V$ in our setting. The main idea is that we will approximate $L$ by the tangent space of a given point, which lies in a slab of sufficiently small scale in the angular direction. Now, let us establish a geometric lemma associated to a fixed point. 
	
	Let $\eta\in L\cap S^{n-1}$,  with ${\rm Ang}(\eta,V^{-})<\f{\pi}{2}-K^{-2}$. Define $\widetilde V$ to be the $(n+1-m)$-dimensional linear subspace spanned by the vectors $\gamma_1,\cdots, \gamma_{n+1-m}$ given by
	$$\gamma_i:={\rm Hess}(\phi)\big|_{\xi=\eta}\alpha_i,\;  \alpha_i=(a_{i,1},\cdots, a_{i,n}), \;i=1,\cdots, n+1-m, $$
	The assumption ${\rm Ang}(\eta,V^{-})<\f{\pi}{2}-K^{-2}$ ensures that $\gamma_i,\ i=1,\cdots, n+1-m$ are linearly independent. Let  $\bar V^{-}$ be the orthogonal complement of $\widetilde V$ in $\R^n$, i.e.
	\beqq
	\R^n=\widetilde V\oplus \bar V^{-}.
	\eeqq
	Let $\bar V$ be the linear subspace spanned by $\bar V^{-}$ and $e_{n+1}$. Define $W$ to be the orthogonal complement space of $\bar V$ in $\R^{n+1}$, i.e.
	\beqq
	\R^{n+1}=\bar V\oplus W.
	\eeqq
	We remark that, unlike the circular cone case, all linear spaces $\widetilde V,\, \bar V,\, \bar V^-,\, W$ defined above depend on the choice of $\eta.$ In fact, we define $\bar V^-$ in such a way that it represents a certain linearization of $L$ at the point $\eta$.
	\begin{lemma}
		Let $\eta \in S^{n-1}\cap L$. If ${\rm Ang}(\eta, V^{-})<\f{\pi}{2}-K^{-2}$, then $W$ and $V$ are transversal in the sense that ${\rm Ang}(V,W)\gtrsim K^{-4}$;
	\end{lemma}
	
	\begin{proof}
		Let $\{\beta_1,\cdots, \beta_{m-1}\}$ be an orthonormal basis for $V^{-}$, we may choose another unit vector $\beta_m$ such that $\{\beta_1,\cdots,\beta_{m-1}, \beta_m\}$ forms an orthonormal basis for $V$.  To prove $${\rm Ang}(V,W)\gtrsim K^{-4},$$
		it suffices to show for any unit vector $v$ parallel to $V$,
		\beq\label{proj1}
		|{\rm Proj}_{\bar V}v|\gtrsim K^{-4}.
		\eeq
		To prove \eqref{proj1}, we subdivide the set vectors of $v$ which are parallel $V$ into two categories.
		
		\noindent{\bf Case IIa: $|{\rm Proj}_{\beta_m}v|>K^{-4}$.}  
		
		Assume that $v=\sum_{i=1}^{m-1}a_i \beta_i+a_m \beta_m,$ then we have $|a_m|>K^{-4}$. Therefore
		$$|{\rm Proj}_{\bar V}v|\ge|{\rm Proj}_{e_{n+1}}v|= |a_m||{\rm Proj}_{e_{n+1}}\beta_m|\gtrsim K^{-4},$$
		where we used the fact that
		$$|{\rm Proj}_{e_{n+1}}\beta_m|\gtrsim 1.$$
		
		\noindent{\bf Case IIb: $|{\rm Proj}_{\beta_m}v|\leq K^{-4}$.}

		In this case 
		$$|{\rm Proj}_{ V^{-}}v|\geq 1/2.$$
		Then \eqref{proj1} will follow from the following claim: 
		
		{ \bf Claim:} for each unit vector $v_1 \in \bar V^{-}, v_2\in V^{-}$,  \beq
		{\rm Ang}(v_1, v_2)\leq \frac{\pi}{2}-C K^{-4}.
		\eeq
		Since $\alpha_i, i=1,\cdots, n+1-m$ form a basis of  a subspace which is orthogonal to $V^{-}$ in $\R^n$,  and  ${\rm Hess}(\phi)\big|_{\xi=\eta} \alpha_i, i=1,\cdots, n+1-m$ form a basis of $\widetilde V$ which is orthogonal to $\bar V^{-}$. Thus it suffices to show
		\beq \label{anglere} \Big\langle \alpha_i, {\rm Hess}(\phi)\big|_{\xi=\eta} \alpha_i\Big\rangle\gtrsim K^{-4}.\eeq
		Note that 
		$${\rm Ang}(\eta, V^{-})<\frac{\pi}{2}-K^{-2},$$
		which implies that 
		\beq \label{tangle}{\rm Ang}(\eta, \alpha_i)>K^{-2},\; i=1,\cdots, n+1-m.\eeq
		Let $\alpha_i=a_1\eta+a_2 \tilde \eta,$ where $\tilde \eta$ lies in the subspace $\eta^{\perp}$, which is the orthogonal complement of $\{t\eta, t\in \R\}$ in $\R^{n}.$
		\eqref{tangle} implies $a_2>cK^{-2}$. Since $\eta\in \ker {\rm Hess}(\phi)\big|_{\xi=\eta},$ We have
		\beq  \Big\langle \alpha_i, {\rm Hess}(\phi)\big|_{\xi=\eta} \alpha_i\Big\rangle=(a_2)^2 \Big\langle \tilde \eta, {\rm Hess}(\phi)\big|_{\xi=\eta} \tilde \eta\Big\rangle\gtrsim K^{-4}.\eeq
		In the last inequality, we have used the fact that ${\rm Hess}(\phi)\big|_{\xi=\eta}$ is nondegenerate when restricted to the subspace $\eta^{\perp}$.
	\end{proof}
	
	We summarize the above discussion below.
	\begin{lemma}\label{eqaaa}
		Let  $L, \bar V, V^{-}$ and $W$ are defined as above and $\eta\in L$. We have the following dichotomy:
		\begin{enumerate}
			\item[{\bf\noindent $a$)}] If 
			${\rm Ang}(\eta,V^{-})\leq \f{\pi}{2}-K^{-2}$, then $W$ and $V$ are transversal in the sense that ${\rm Ang}(V,W)\gtrsim K^{-4}$;
			\item[{\bf \noindent $b$) }] If 
			${\rm Ang}(\eta,V^{-})>\f{\pi}{2}-K^{-2}$, then $L$ is contained in a slab of dimensions   $\sim 1\times K^{-2}\times \cdots \times  K^{-2}.$
		\end{enumerate}
	\end{lemma}

	\section{$k$-broad ``norm" estimate}\label{section-5}
	
	In this section, we prove $k$-broad ``norm" estimates associated to a general phase function $\phi$ in the class $\mathbf \Phi$. As discussed earlier, the same estimates should hold for any $\phi_R\in\mathbf \Phi(R)$.  Recall that we have $1\ll K\ll R^{\varepsilon}$. We partition ${\rm A}(1)$, in the angular direction,  into a collection of slabs $\nu$ and $\tau$ of dimensions $1\times R^{-1}\times \cdots \times R^{-1}$ and $1\times K^{-1}\times \cdots \times  K^{-1}$ respectively. In this part, we write  $f_\nu:=f\chi_\nu$ and choose $V\subset \R^{n+1}$ to be a $(k-1)$-dimensional linear subspace. The set $G(\nu)$ consisting of the unit normal vectors of the cone associated with the slab  $\nu$ is defined by
	$$G(\nu):=\bigg\{\f{1}{\sqrt{1+|\nabla\phi|^2}}\big(-\nabla\phi(\xi),1\big): \xi \in \nu\bigg\}.$$
	Similarly, we define 
	\beqq
	G(\tau):=\bigcup_{\nu\subset \tau}G(\nu).
	\eeqq
	We denote by  ${\rm Ang}(G(\nu), V)$ the smallest angle between non-zero vectors $v\in V$ and $v'\in G(\nu)$.
	
	For each $B^{n+1}_{K^2}\subset C_R^{n+1}$, we define 
	$\mu_{Ef}(B^{n+1}_{K^2})$ by
	\beqq
	\mu_{Ef}(B_{K^2}^{n+1}):=\min \limits_{V_1,\cdots, V_{A}}\max\limits_{\substack{\tau:\,\forall 1\leq \ell \leq A \\{\rm Ang}(G(\tau),V_{\ell})>K^{-2}}}\Big( \int_{B^{n+1}_{K^2}} |Ef_\tau|^pdxdt\Big).
	\eeqq

	Let  $\{B_{K^2}^{n+1}\}$ be a collection of finitely overlapping balls which forms a cover of $C_R^{n+1}$.
	Then we define the $k$-broad ``norm" as
	\beqq
	\big\|Ef\big\|_{{\rm BL}_{k,A}^p(C_R^{n+1})}^p:=\sum_{B^{n+1}_{K^2}\subset C_R^{n+1}} \mu_{Ef}(B^{n+1}_{K^2}).
	\eeqq
	Next, we will record some useful properties of the broad ``norm" of which the proof can be found in \cite{Guth}.
	\begin{lemma}[Triangle inequality]
		Suppose that $1\leq p<\infty$, $f=g+h$ and $A=A_1+A_2$, where $A, A_1,A_2$ are nonnegative integers. Then
		\beq\label{BT}
		\|Ef\|_{\br_{k,A}^p(U)}\lesssim \|Eg\|_{\br_{k,A_1}^p(U)}+\|Eh\|_{\br_{k,A_2}^p(U)}.
		\eeq
	\end{lemma}

	\begin{lemma}[H\"older's inequality]
		Suppose that $1\leq p,\,p_1,\,p_2<\infty$, $0\leq \alpha_1,\,\alpha_2\leq 1$ satisfy $\alpha_1+\alpha_2=1$ and
		$$\f{1}{p}=\f{1}{p_1}+\f{1}{p_2}.$$
		Suppose that $A=A_1+A_2$, then
		\beq
		\|Ef\|_{\br_{k,A}^p(U)}\leq \|Ef\|_{\br_{k,A_1}^p(U)}^{\alpha_1} \|Ef\|_{\br_{k,A_2}^p(U)}^{\alpha_2}.
		\eeq
	\end{lemma}
	In the argument we will need to choose $A$ sufficiently large to ensure that the above inequalities may be applied quite a few (but finitely many) times. At the end of the argument, the reader shall see that the relation between the parameters $K,A,R$ can be described by the following inequalities:
	$$1\ll A\lesssim K^\varepsilon\lesssim R^{\varepsilon^2}.$$

	In this part, we aim to prove that the following broad-``norm'' estimate.
	
	\begin{theorem}\label{Theorem-5}
		For any $2\leq k\leq n+1$ and $\varepsilon>0$,  there exists a large constant $A$ such that
		\beq\label{eq-211}
		\|Ef\|_{{\rm BL}_{k,A}^p(C_R^{n+1})}\lesssim_{\varepsilon,\phi}R^\varepsilon \|f\|_{L^2({\rm A}(1))}, \;\;\;{\rm supp}\;f\subset {\rm A}(1),
		\eeq
		for $p\geq  2\frac{n+k+1}{n+k-1}$, where the implicit constant depends on the derivatives of $\phi$ up to finite orders and the eigenvalues of the hessian matrix of $\phi$.
	\end{theorem}
	
	As a direct consequence of Theorem \ref{Theorem-5} and  Lemma \ref{lemma-1},
	we have the following $L^p$ estimate. 
	\begin{corollary}\label{cor2}
		For any $2\leq k\leq n+1$ and $\varepsilon>0$,
		there is a large constant $A$ such that
		\beq\label{eq-29}
		\|e^{it\phi(D)}f\|_{{\rm BL}_{k,A}^p(C_R^{n+1})}\lesssim_{\varepsilon}R^{n(\f{1}{2}-\f{1}{p})+\varepsilon} \|f\|_{L^p(\R^n)},\;\;{\rm supp}\;\widehat{f}\subset {\rm A}(1),
		\eeq
		for $p\geq 2\frac{n+k+1}{n+k-1}$.
	\end{corollary}
	\begin{proof}
		We decompose $f$ in spatial space, and obtain by Lemma \ref{lemma-1}
		$$\big|e^{it\phi(D)}f(x)\big|\le \big|e^{it\phi(D)}(\Psi_{B_{R^{1+\varepsilon}}^n}f)\big|+{\rm RapDec}(R)\|f\|_{L^p}.$$
		Then \eqref{eq-29} follows from Theorem \ref{Theorem-5} and  H\"older's inequality.
	\end{proof}
	To prove Theorem \ref{Theorem-5}, we need to employ the polynomial partitioning argument. Let us first collect some useful results from \cite{Guth} and \cite{OW}.
	
	\begin{definition}[Transverse complete intersection]
		Fix an integer $m\in [1,n+1]$ and let $P_1,\cdots,P_{n+1-m}$ be polynomials on $\R^{n+1}$ whose common zero set is denoted by $Z(P_1,\cdots, P_{n+1-m})$. We use $D_Z$ to denote the degree of $Z(P_1,\cdots, P_{n+1-m})$ which is the highest degree among those polynomials $P_i$'s. The variety $Z(P_1,\cdots, P_{n+1-m})$ is called a transverse complete intersection if
		$$\nabla P_1(x)\wedge \cdots \wedge \nabla P_{n+1-m}(x)\neq 0, \forall x\in Z(P_1,\cdots, P_{n+1-m}).$$
		
	\end{definition}
	
	The following theorem is essentially proved in Section 8.1 of \cite{Guth} while not explicitly stated there. 
	\begin{theorem}[\cite{Guth}]\label{poly}
		Let $r\gg 1, d\in \mathbb{N}$ and  $F\in L^1(\R^n)$ be non-negative and supported on $B_r^n\cap N_{r^{1/2+\delta}}Z$ for some $0<\delta \ll 1$, where $Z$ is an $m$-dimensional transverse complete intersection of degree $D_Z=d$. Then, at least one of the following cases holds: 
		\begin{description}
			\item[(1) Cellular case] There exists a polynomial $P: \R^n \rightarrow \R$ of degree $D=D(d)$ such that there exists $\sim D^m$ cells $O_i\subset Z\backslash N_{r^{1/2+\delta}}Z(P)$ with $O_i \subset B_{r/2}^n$ and 
			\beqq
			\int_{O}F\sim D^{-m}\int_{\R^n}F, \quad \text{for all}\;\; O.
			\eeqq
			Furthermore, each tube of length $r$ and radius $r^{1/2+\delta}$ intersects at most $O(D)$ cells.
			\item[(2) Algebraic case] There exists an $(m-1)$-dimensional transverse complete intersection $Y$ of degree at most $O(D)$ such that 
			\beqq
			\int_{B_r^n\cap N_{r^{1/2+\delta}}Z}F\lesssim \int_{B_r^n\cap N_{r^{1/2+\delta}}Y}F.
			\eeqq
		\end{description}
		
	\end{theorem}
	
	\begin{proposition}
		Let $T$ be a cylinder of radius $r$ with central line $\ell$ and suppose that $Z=Z(P_1,\cdots, P_{n+1-m})\subset \R^{n+1}$ is a transverse complete intersection, where the polynomials $P_j$ have degree at most $D$. For any $\alpha$, define
		$$Z_{>\alpha}:=\{z\in Z: {\rm Angle}(T_zZ,\ell)>\alpha\}.$$
		Then $Z_{>\alpha}\cap T$ is contained in a union of $\lesssim D^n$ balls of radius $\lesssim r\alpha^{-1}$.
	\end{proposition}
	\begin{definition}\label{tangto}
		Let $Z$ be an $m$-dimensional variety in $\R^{n+1}$. A tube $T_{\nu,w}^\ell$ is said to be $\gamma$-tangent to $Z$ in $B_{R}^{n+1}$ if 
		$$T_{\nu,w}^\ell \subset N_{\gamma R}(Z)\cap B_{R}^{n+1},$$
		and for all $z\in Z\cap N_{\gamma R}(T_{\nu,w}^\ell)$ there holds 
		$${\rm Ang}(T_zZ, {\bf L}(\nu))\leq \gamma.$$
	\end{definition}

	\begin{definition}\label{defcon}
		Let $Z$ be a transverse complete intersection of degree $D\sim O(1)$ and dimension $m$ inside $B_R^{n+1}$. Define
		\beqq
		\mathbb{T}_Z:=\{(\nu,w,\ell): T_{\nu,w}^\ell\; \text{is}\; R^{-1/2+\delta_m}\text{-tangent to } Z\; \text{in}\; B_R^{n+1}\},
		\eeqq
		where $\delta_m\geq 0$ is a fixed small parameter for each dimension $m$.
	\end{definition}

	Theorem \ref{Theorem-5} can be deduced from the following proposition.
	\begin{proposition}\label{pro12}
		For $\varepsilon>0$, there are small parameters
		$$0<\delta\ll \delta_n\ll\delta_{n-1}\ll\cdots \ll \delta_1\ll \delta_0\ll \varepsilon,$$
		and a large constant $\bar A$ such that the following holds. Let $1\leq m\leq n+1$ and $Z=Z(P_1,\cdots, P_{n+1-m})$ be a transverse complete intersection with degree  $D_Z$. Suppose that $f$ is concentrated on a union of wave packets coming from $\mathbb{T}_{Z}$. Then, for any $1\leq A\leq \bar A$ and radius $ R\geq 1$, we have
		\beq \label{inmain}
		\|Ef\|_{{\rm BL}_{k,A}^p(C_R^{n+1})}\leq C(K,\varepsilon,m,D_Z)R^{m\varepsilon}R^{\delta(\log{\bar A}-\log A)}R^{-e+1/2}\|f\|_{L^2}
		\eeq
		for all
		\beqq
		2\leq p\leq \bar p(k,m),
		\eeqq
		where
		\beqq
		e:=\f{1}{2}\big(\f{1}{2}-\f{1}{p}\big)(n+1+k)
		\eeqq
		and 
		\beqq
		\bar p(k,m):=\begin{cases}
			2\dfrac{m+k}{m+k-2},\; &k<m;\\
			\dfrac{2m}{m-1}+\delta,\;& k=m.
		\end{cases}
		\eeqq
	\end{proposition}
	
	When $m=n+1$, by taking $Z=\R^{n+1}$ and choosing $A=\bar A$ and $p=\bar p(k,n+1)$, we have $-e+\f{1}{2}=0$. Therefore, Theorem \ref{Theorem-5} follows from Proposition \ref{pro12}.
	
	For $p=2$, Proposition \ref{pro12} follows directly from the trivial $L^2$ estimate:
	$$\|Ef\|_{L^2(C_R^{n+1})}\leq CR^{1/2}\|f\|_{L^2}.$$
	Thus, by interpolation and H\"older's inequality of the broad norm, Proposition \ref{pro12} is reduced to the endpoint case $p=\bar p(k,m)$. We  prove Proposition \ref{pro12} by an induction argument. In particular, we will induct on the dimension $m$, the radius $R$, and the parameter $A$. We start by checking the base case of the induction. If $R$ is small, the desired estimate can be deduced by choosing  $C(K,\varepsilon,m,D_Z)$ sufficiently large. If $A=1$, we may choose $\bar A$ sufficiently large such that $R^{\delta(\log \bar A-\log 1)}=R^{10n}$, then the desired estimate will follow  from the following trivial estimate
	$$\|Ef\|_{\br_{k,1}^p(C_R^{n+1})}\leq |C_R^{n+1}|\|f\|_{L^2}.$$
	Finally, we will check the base case for the dimension $m$. This can be deduced from the following lemma from \cite{Guth, OW}. 
	\begin{lemma}[\cite{Guth, OW}]\label{le3}
		If $Ef$ is $R^{-1/2+\delta}$-tangent to a variety $Z$ of degree $O(1)$ and dimension $m\leq k-1$, then
		\beq
		\|Ef\|_{\br_{k,A}^p(B_R^{n+1})}\leq {\rm RapDec}(R)\|f\|_{L^2}.
		\eeq
	\end{lemma}
	If $m=k-1$, then by Lemma \ref{le3}, we have
	\beqq
	\|Ef\|_{\br_{k,A}^p(C_R^{n+1})}\leq {\rm RapDec}(R)\|f\|_{L^2}.
	\eeqq
	
	Next, we assume that Proposition \ref{pro12} holds if we decrease the dimension $m$, the radius $R$, or the value of $A$. We proceed the inductive steps.
	
	By invoking Theorem \ref{poly} with \[F=\frac{1}{|B_{K^2}^{n+1}|}\sum_{B_{K^2}^{n+1}\subset C_R^{n+1}}\mu_{Ef}(B_{K^2}^{n+1})\chi_{B_{K^2}^{n+1}},\] we see that there are two cases, that is, either the mass of $\mu_{Ef}$ can be concentrated in a small neighborhood of a lower-dimensional variety or we can reduce the estimate of $\mu_{Ef}$ to smaller cells. We say we are in the \emph{algebraic case}, if there is a transverse complete intersection $Y \subset Z$ of dimension $m-1$, defined by using polynomials of degree $\leq D(\varepsilon,D_Z)$ with
	\beq
	\mu_{Ef}\Big(N_{R^{1/2+\delta_m}}(Y)\cap C_R^{n+1}\Big)\gtrsim \mu_{Ef}(C_R^{n+1}).
	\eeq
	Otherwise, we say we are in the \emph{cellular case} with
	\beq
	\mu_{Ef}(C_R^{n+1})\lesssim \sum_{i=1}^{ \sim D^m} \mu_{Ef}(O_i).
	\eeq
	\subsection{The cellular case}
	Assume that we are in cellular case, that is $\sum_i\mu_{Ef}(O_i)\sim \mu_{Ef}(C_R^{n+1})$. For a given $i$, define $f_i=\sum_{(\nu,w,\ell)\in \mathbb{T}_i} f_{\nu,w}^\ell$, where
	$$\mathbb{T}_i:=\{(\nu,w,\ell):T_{\nu,w}^\ell\cap O_i\neq \varnothing\}.$$
	By a pigeonholing argument, we may choose a cell $O_i$ such that
	
	\begin{equation}
		\begin{aligned}
			\mu_{Ef}(C_R^{n+1})\lesssim D^m \mu_{Ef}(O_i),\\
			\|f_i\|_{L^2}^2\lesssim \f{1}{D^{m-1}}\|f\|_{L^2}^2.
		\end{aligned}
	\end{equation}
	By covering $O_i$ by a family of finitely-overlapping balls of radius $R/2$,  we can prove \eqref{inmain} by inducting on $R$ as follows
	\beq
	\begin{aligned}
		\mu_{Ef}(C_R^{n+1})&\lesssim D^m\mu_{Ef}(O_i)\lesssim D^m\sum_{B_{R/2}^{n+1}\subset C_R^{n+1}}\mu_{Ef}(O_i\cap B_{R/2}^{n+1})\\
		&\lesssim R^\varepsilon D^m \|f_i\|_{L^2}^p\lesssim R^\varepsilon D^{m-\f{(m-1)p}{2}}\|f\|_{L^2}^p.
	\end{aligned}
	\eeq
	The induction closes for $p>\f{2m}{m-1}$ if we choose $D(\varepsilon, D_Z)$ sufficiently large to control the implicit constant.

	\subsection{The algebraic case }
	By definition, there exists a transverse complete intersection $Y$ of dimension $m-1$ such that
	\beqq
	\mu_{Ef}(N_{R^{1/2+\delta_m}}(Y))\gtrsim \mu_{Ef}(C_R^{n+1}).
	\eeqq
	In this case, we subdivide $C_R^{n+1}$ into  balls $\{B_j\}_j$ of radius $\rho$, with $R^{1/2}\ll \rho\ll R$ and  $\rho^{1/2+\delta_{m-1}}=R^{1/2+\delta_m}$.
	Define $f_j=\sum_{(\nu,w,\ell)\in \mathbb{T}_j}f_{\nu,w}^\ell$, where
	$$\mathbb{T}_j:=\{(\nu,w,\ell):T_{\nu,w}^\ell \cap N_{R^{1/2+\delta_m}}(Y)\cap B_j\neq \varnothing\}.$$
	We further decompose $\mathbb{T}_j$ into tubes that are tangential to $Y$ and tubes that are transverse to $Y$. We say $T_{\nu,w}^\ell$ is tangent to $Y$ in $B_j$ if 
	\beqq
	T_{\nu,w}^\ell\cap 2B_j\subset N_{R^{1/2+\delta_m}}(Y)\cap 2B_j=N_{\rho^{1/2+\delta_{m-1}}}\cap 2B_j.
	\eeqq
	and for any $x\in T_{\nu,w}^\ell$ and $Y\cap 2B_j$ with $|x-y|\lesssim R^{1/2+\delta_m}=\rho^{1/2+\delta_{m-1}}$, $${\rm Ang}(G(\nu), T_y Y)\lesssim \rho^{-1/2+\delta_{m-1}}.$$
	Define the collection of tangential wave packets $\mathbb{T}_{j,\ta}$ as  
	$$\mathbb{T}_{j,\ta}=\{(\nu,w,\ell)\in \mathbb{T}_j: T_{\nu,w}^\ell \; \text{is tangent to}\; Y \;\text{in}\; B_j\},$$
	and the transverse wave packets $\mathbb{T}_{j,\tra}$ by
	$$\mathbb{T}_{j,\tra}:=\mathbb{T}_j\backslash \mathbb{T}_{j,\ta}.$$
	Correspondingly, we define
	$$f_{j,{\rm tang}}:=\sum_{(\nu,w,\ell)\in \mathbb{T}_{j,{\rm tang}}} f_{\nu,w}^\ell,\;f_{j,{\rm trans}}:=\sum_{(\nu,w,\ell)\in \mathbb{T}_{j,{\rm trans}}}f_{\nu,w}^\ell.$$
	Note that $Ef$ is essentially equal to  $Ef_j$ on the ball $B_j$ in the sense that $Ef=Ef_j+{\rm RapDec}(R)\|f\|_{L^2}$. By the triangle inequality, we have
	\beq
	\sum_{j}\|Ef\|_{{\rm BL}_{k,A}^p(B_j)}^p\lesssim \sum_j\|Ef_{j,{\rm tang}}\|_{{\rm BL}_{k,A/2}^p(B_j)}^p+\sum_j\|Ef_{j,{\rm trans}}\|_{{\rm BL}_{k,A/2}^p(B_j)}^p+{\rm RapDec}(R)\|f\|^p_{L^2}.
	\eeq
	Therefore, it remains to prove Proposition \ref{pro12} for both the tangential case and the transversal case.
	\subsection{The tangential case}
	Assuming the tangential part dominates, we will prove Proposition \ref{pro12} by induction on the dimension $m$ and $A$. Since we are now working with a ball of radius $\rho\ll R$, in order to match our assumption, we need to redo the wave packet decomposition at scale $\rho$. For the sake of simplicity, define $g=f_{j,{\rm tang}}$ and
	\beqq
	\tilde{g}=\sum_{\tilde \nu,\tilde w,\tl}\tilde{g}_{\tilde \nu,\tilde w}^{\tilde{\ell}}+{\rm RapDec}(R)\|f\|_{L^2}.
	\eeqq
	In order to perform the induction on dimension argument, we have to verify that $(\tilde \nu,\tilde w,\tl)$ corresponds to tubes which are tangent to $Y$ in $B_j$. To be more precise, we need to show that
	$$T_{\tilde \nu,\tilde w}^{\tilde{\ell}} \subset N_{\rho^{1/2+\delta_{m-1}}}(Y)\cap B_j,$$
	and for any $x\in T_{\tilde \nu,\tilde w}^{\tilde{\ell}}$, $y\in Y\cap B_j$ with $|x-y|\lesssim \rho^{1/2+\delta_{m-1}}$,
	$${\rm Ang}(G(\tilde \nu), T_{y}Y)\lesssim \rho^{-1/2+\delta_{m-1}},$$
	which can be deduced from Lemma \ref{tanc}.
	By induction on $m$ and $A$, we have
	\begin{equation}
		\begin{aligned}
			\|E\tilde g\|_{{\rm BL}_{k,A/2}^p(B_\rho)}\leq C(K,\varepsilon,m,D(\varepsilon,D_Z))\rho^{(m-1)\varepsilon}\rho^{\delta(\log{\bar A}-\log{A/2})}\rho^{-e+1/2}\|f_{j,{\rm tang}}\|_{L^2}.
		\end{aligned}
	\end{equation}
	for
	$$2\leq p\leq \bar{p}(k,m-1).$$
	
	Since there are $R^{O(\delta_{m-1})}$ many $B_j$'s, by summing over the balls and noting that
	$$\rho^{1/2+\delta_{m-1}}=R^{1/2+\delta_m},$$
	finally we have 
	\begin{equation}
		\begin{aligned}
			\|E\tilde g\|_{{\rm BL}_{k,A/2}^p(C_R^{n+1})}\leq C(K,\varepsilon,m,D(\varepsilon,D_Z))R^{O(\delta_{m-1})}R^{(m-1)\varepsilon}R^{\delta(\log{\bar A}-\log{A})}\rho^{-e+1/2}\|f_{j,{\rm tang}}\|_{L^2}.
		\end{aligned}
	\end{equation}
	Using the fact that  $\delta_{m-1}\ll \varepsilon$,  we  close the induction.
	
	\subsubsection{The transversal case}\label{equid section}
	Unlike the circular cone case studied in \cite{OW}, $\{(\xi,\phi(\xi)): \xi\in L\}$ may not lie in an affine subspace. To overcome this difficulty, we  work with small sector $\tau$ of dimension $\rho^{-1/2+\delta_m}\times \cdots \times \rho^{-1/2+\delta_m}\times 1$ with $R^{1/2}\ll \rho\ll R$. At this scale, an affine subspace $\bar V$ can be constructed using the tangent space of a point on the surface $\{(\xi,\phi(\xi)): \xi\in L\}$, so that 
	$\{(\xi, \phi(\xi)):\xi \in L\cap \tau\}$ lies in a $R^{-1/2+\delta_m}$-neighborhood of $\bar V$.
	
	Given $\eta\in L$, we use $T_\eta L$ to denote the tangent space of $L$ at $\eta$. By the definition of $L$, it is easy to check that $T_\eta L$ is orthogonal to the vectors 
	\beqq
	\Big(\f{\partial ^2\phi}{\partial \xi_i\partial \xi_j}\Big)\Big|_{\xi=\eta}\alpha_i,\;i=1,\cdots, n+1-m.
	\eeqq
	Let $$\bar V:=T_{\eta}L+e_{n+1}.$$ 
	This definition of $\bar V$ is consistent with that in subsection \ref{gem}.
	\begin{lemma}\label{apt}
		Let $\tau$ be a cap of dimension $1\times \rho^{-1/2+\delta_m}\times \cdots,\times \rho^{-1/2+\delta_m}$ with $\eta\in \tau$, then 
		\beq \label{eq-quar}
		\{(\xi,\phi(\xi)): \xi\in L\cap \tau\}\subset N_{CR^{-1/2+\delta_m}} \bar V.
		\eeq
		where $C>0$ is a large constant.
	\end{lemma} 
	\begin{proof}
		To prove \eqref{eq-quar}, it suffices to show 
		\beq
		\{\xi: \xi\in L\cap \tau\}\subset N_{CR^{-1/2+\delta_m}} T_{\eta}L.
		\eeq
		By the homogeneity of $\phi$, this can be further reduced to showing that if $\xi,\eta \in L\cap S^{n-1}$ with $|\xi-\eta|\lesssim \rho^{-1/2+\delta_m}$, then 
		\beq
		\{\xi: \xi\in L\cap S^{n-1}\cap \tau\}\subset N_{CR^{-1/2+\delta_m}} T_{\eta}L.
		\eeq
		Since $|\xi-\eta|\lesssim \rho^{-1/2+\delta_m}$ we may construct a curve $\{\xi(t)\}\subset L\cap S^{n-1}$  connecting $\xi$ and $\eta$ with $\xi(0)=\eta, \xi(\rho^{-1/2+\delta_m})=\xi$ and 
		\beqq
		|\xi^{\ell}(t)|\leq C, \; 0\leq t\leq \rho^{-1/2+\delta_m},\; \ell \leq 3.
		\eeqq
		It remains to show 
		\beqq
		\Big|{\rm Proj}_{\big(T_\eta (L\cap S^{n-1}) \big)^{\perp}}(\xi(t)-\xi(0))\Big|\lesssim R^{-1/2+\delta_m}.
		\eeqq
		Note that 
		\beqq
		\Big(\frac{\partial^2\phi}{\partial \xi_i\xi_j}\Big)\Big|_{\xi=\eta}\alpha_i,\; i=1,\cdots, n+1-m, 
		\eeqq
		are orthogonal to $T_\eta (L\cap S^{n-1})$, and by our construction $\xi'(0)\in T_{\eta}(L\cap S^{n-1})$. Therefore we have 
		\beqq
		\Big|(\xi(\rho)-\xi(0))\cdot \Big(\frac{\partial^2\phi}{\partial \xi_i\xi_j}\Big)\Big|_{\xi=\eta}\alpha_i\Big|= \rho\xi'(0)\cdot  \Big(\frac{\partial^2\phi}{\partial \xi_i\xi_j}\Big)\Big|_{\xi=\eta}\alpha_i + O(\rho^{-1+2\delta_m})\lesssim R^{-1/2+\delta_m}.
		\eeqq
		Here we have used the assumption $R^{1/2}\ll \rho \ll R$.
		The proof is complete.

	\end{proof} 
	
	
	Fix a ball $B$ of radius $R^{1/2+\delta_m}$. Let $V$ be the tangent space to $Z$ at some point in $B\cap Z$. By Definition \ref{tangto}, it does not matter which point we choose.  Define two sets $\mathbb{T}_{B,Z}$ and $\mathbb{T}_{B,Z,\tau}$ respectively as follows:\beqq
	\mathbb{T}_{B,Z}:=\{(\nu,w,\ell): T_{\nu,w}^\ell \;\text{is }\;R^{-1/2+\delta_m}\;\text{tangent to }\;  Z, T_{\nu,w}^\ell\cap B\neq \varnothing\},
	\eeqq
	\beqq
	\mathbb{T}_{B,Z,\tau}:=\{(\nu,w,\ell): T_{\nu,w}^\ell \;\text{is }\;R^{-1/2+\delta_m}\;\text{tangent to }\;  Z, T_{\nu,w}^\ell\cap B\neq \varnothing, \nu\cap 2\tau \neq \varnothing\}.
	\eeqq
	Let $h_B$ related to $\mathbb{T}_{B,Z}$ be defined by 
	$$h_B:=\sum_{(\nu,w,\ell)\in \mathbb{T}_{B,Z}}h_{\nu,w}^\ell.$$
	Similarly, define $h_{B,\tau}$ related to $\mathbb{T}_{B,Z,\tau}$ as  $$h_{B,\tau}:=\sum_{(\nu,w,\ell)\in \mathbb{T}_{B,Z,\tau}}h_{\nu,w}^\ell.$$
	
	Fix $B_j=B_\rho^{n+1}(y)$ and cover $B_j$ by balls $B$ of radius $R^{1/2+\delta_m}$. Since $V$ is determined by $B$, on account of Lemma \ref{eqaaa}, we may sort the balls $B$ into two classes $X_a$ and $X_b$ according to whether case $a)$ or case $b)$ in Lemma \ref{eqaaa} holds.
	Now partition $N_{R^{1/2+\delta_m}}(Z)\cap B_j\subset X_a\cup X_b$, where $X_a, X_b$ are the union of balls $B$ in case $a)$ or in case $b)$ respectively.
	First assume that $B\in X_b$. Since the support of $h_B$ is contained in $O(1)$ slabs, we have
	\beq \label{tangt} \|Eh_{B}\|_{{\rm BL}_{k,A}^p(B)}^p={\rm Rapdec}(R)\|h_B\|_{L^2}.
	\eeq
	Otherwise,  we have 
	\begin{lemma}
		Let $h_{B,\tau}=\sum_{(\nu,w,\ell)\in \mathbb{T}_{B,Z,\tau}} h_{\nu,w}^\ell$ and $B\in  X_a$. Then for any $\rho\leq R$,
		\beq\label{equi}
		\int_{B\cap N_{\rho^{1/2+\delta_m}}(Z)}|Eh_{B,\tau}|^2\lesssim R^{O(\delta_m)}\Big(\f{R^{1/2}}{\rho^{1/2}}\Big)^{-(n+1-m)}\int_{2B}|Eh_{B,\tau}|^2+{\rm Rapdec}(R)\|h_B\|_{L^2}^2.
		\eeq
	\end{lemma}
	\begin{proof}
		Since $V$ is the tangent space of $Z$ at some point in $B\cap Z$, 
		\beqq
		T_{B,Z,\tau}\subset T_{B,V}:=\{(\nu,w,\ell): T_{\nu,w}^\ell\cap B\neq \varnothing \;\text{and}\; {\rm Ang}({\bf L}(\nu),V)\lesssim R^{-1/2+\delta_m}\}.
		\eeqq 
		By Lemma \ref{apt}, $(Eh_{B,\tau})^{\wedge}$ is supported in $N_{R^{-1/2+\delta_m}} \bar V$. Consider an $(n+1-m)$-dimensional plane $\Pi$ parallel to $W$ passing through $B$. If we restrict $Eh_{B,\tau}$ to the plane $\Pi$, then its Fourier transform is supported in a ball of radius $\lesssim R^{-1/2+\delta_m}$. Therefore, by Lemma 6.4 in \cite{Guth} we have 
		\beqq
		\int_{B(\bar x,\rho^{1/2+2\delta_m})\cap \Pi} |Eh_{B,\tau}|^2\lesssim \Big(\f{R^{1/2-2\delta_m}}{\rho^{1/2+2\delta_m}}\Big)^{-(n+1-m)}\int_{\Pi} w_{B(\bar x,R^{1/2-2\delta_m})}|Eh_{B,\tau}|^2,
		\eeqq
		for any point $\bar x \in \R^{n+1} $.
		
		By Lemma \ref{eqaaa}, ${\rm Ang}(V,W)\gtrsim K^{-4}$, we have 
		\beqq
		\Pi\cap N_{\rho^{1/2+\delta_m}}(Z)\cap B\subset \Pi\cap B(x_0, \rho^{1/2+2\delta_m}),
		\eeqq  
		for some point in $B$.
		
		Therefore, modulo a rapidly decaying error, we obtain
		\beqq
		\int_{\Pi\cap N_{\rho^{1/2+\delta_m}}(Z)\cap B}|Eh_B|^2\leq \int_{\Pi\cap B(x_0,\rho^{1/2+2\delta_m})}|Eh_B|^2\lesssim R^{O(\delta_m)}\Big(\f{R^{1/2}}{\rho^{1/2}}\Big)^{-(n+1-m)}\int_{\Pi}w_B |Eh_B|^2.
		\eeqq 
		Integrating over all $\Pi$ that is parallel to $W$ and passing through $B$, one obtains the desired results.  
		
	\end{proof}

	Define $g_{{\rm ess}}$ and $g_{{\rm tail}}$ to be the essential part and tail part of $f_{j,\trans}$ respectively by
	\beqq
	g_{\ess}:=\sum_{(\tn_0,w_0)\in \mathbb{T}_{\ess}} g_{\tn_0,w_0},
	\eeqq
	\[ g_{\rm tail}:=\sum_{(\tn_0,w_0)\in \mathbb{T}_{\rm tail}} g_{\tn_0,w_0},\]
	where
	\begin{align*}
		\mathbb{T}_{\ess}&:=\{(\tn_0,w_0): \exists (\nu,w,\ell)\in \mathbb{T}_{\tn_0,w_0}\;\textrm{with}\; T_{\nu,w}^\ell\cap X_a\neq \varnothing\},\\
		\mathbb{T}_{{\rm tail}}&:=\{(\tn_0,w_0):\forall (\nu,w,\ell)\in \mathbb{T}_{\tn_0,w_0},T_{\nu,w}^\ell\cap X_a=\varnothing\}.
	\end{align*}
	
	By the triangle inequality and \eqref{tangt}, we have
	\begin{align*}
		\|Eg\|_{{\rm BL}_{k,A}^p(B_j)}&\leq \|Eg_{\ess}\|_{{\rm BL}_{k,A/2}^p(B_j)}+\|Eg_{{\rm tail}}\|_{{\rm BL}_{k,A/2}^p(B_j)}\\
		&\leq \|Eg_{\ess}\|_{{\rm BL}_{k,A/2}^p(B_j)}+{\rm RapDec}(R)\|f\|_{L^2}.
	\end{align*}
	
	Next, through an appropriate reduction, it suffices to consider a  direction $b$ such that $|b|\le R^{1/2+\delta_m}$ and $b$ is transversal to $T_zZ$ for all points in $z\in Z\cap B_j$. Indeed, we will show that $L^2$-norm of $g_{\ess}$ is equidistributed along different choices of $b$ in $N_{R^{1/2+\delta_m}}(Z)\cap B_j$. To this end, we need a useful reversed H\"ormander's $L^2$ bound which can be found in \cite{Guth}.
	\begin{lemma}[Lemma 3.4 in \cite{Guth}]\label{le5}
		Suppose that $h$ is a function concentrated on a set of wave-packets $\mathbb{T}$ and for every $T_{\nu,w}^\ell\in\mathbb{T}$, $T_{\nu,w}^\ell \cap B_r(z)\neq \varnothing$ for some radius $r\geq R^{1/2+\delta_m}$. Then
		\beqq
		\|Eh\|_{L^2(B_{10r}^{n+1}(z))}^2\sim r\|h\|_{L^2}^2.
		\eeqq
	\end{lemma}
	As a direct consequence of Lemma \ref{le5}, it follows that for any $B\subset X_a$ such that $B\cap T_{\nu,w}^\ell\neq \varnothing$, where $(\nu,w,\ell)\in \mathbb{T}_{\tn_0,w_0}$ for some $(\tn_0,w_0)\in \mathbb{T}_{\ess},$ we have
	\beqq
	\|g_{\tn_0,w_0}\|_{L^2}^2\sim R^{-1/2-\delta_m}\|Eg_{\tn_0,w_0}\|_{L^2(40B)}^2.
	\eeqq
	Let $b\in B_{R^{1/2+\delta_m}}$. Decompose
	\beqq
	\tilde g=\sum_{\tn,\tw,\tl}\tilde g_{\tn,\tw}^{\tl}+{\rm RapDec}(\rho)\|f\|_{L^2}.
	\eeqq
	A key observation is that, for any $(\tn,\tw,\tl)$, if $\tilde{T}_{\tn,\tw}^{\tl}$ intersects $N_{\rho^{1/2+\delta_m}}(Z+b)\cap B_j$, then according to Lemma \ref{tanc}, $\tilde{T}_{\tn,\tw}^{\tl}$ is $\rho^{-1/2+\delta_m}$-tangent to $Z+b$ in $B_j$. Define
	\beqq
	\tilde{\mathbb{T}}_{Z+b}:=\{(\tn,\tw,\tl): \tilde{T}_{\tn,\tw}^{\tl}\; \textrm{is tangent to}\; Z+b \;\text{in}\; B_j\}, \tilde g_b:=\sum_{(\tn,\tw,\tl)\in \tilde T_{Z+b}}\tilde g_{\tn,\tw}^{\tl}.
	\eeqq
	Define
	\beqq
	\tilde g_{{\rm ess},b}=\sum_{(\tn_0,w_0)\in \mathbb{T}_{\ess}}\sum_{(\tn,\tw,\tl)\in \tilde T_{Z+b}\cap \widetilde{\mathbb{T}}_{\tn_0,w_0}} \tilde g_{\tn,\tw}^{\tl}.
	\eeqq
	Therefore, $\tilde g_{{\rm ess},b}$ is tangent to $Z+b$ in $B_j$.
	
	With the above notations, to finish the proof of the transverse case, we need the following important transverse equidistribution estimate, the proof of which can be obtained by carrying over the proof of Lemma 5.13 in \cite{OW}.
	\begin{lemma}
		Let $g_{\ess}$ and $\tilde g_{{\rm ess},b}$ be defined as above, then
		\beqq
		\|\tilde g_{\ess, b}\|_{L^2}^2\leq R^{O(\delta_m)}\Big(\frac{R^{1/2}}{\rho^{1/2}}\Big)^{-(n+1-m)}\|g_{\ess}\|_{L^2}^2.
		\eeqq
	\end{lemma}
	Following the approach in \cite{Guth}, we may  choose a finite set of vectors $\mathcal{B} = \{b\}$  where $b\in B_{R^{1/2+\delta_m}}$ such that
	for each $B_j$, we have
	\beqq
	\|Eg_{\ess}\|_{{\rm BL}_{k,A/2}^p(B_j)}^p\lesssim (\log R)\sum_{b\in\mathcal B}\|Ef_{j,\trans,b}^{\ess}\|_{{\rm BL}_{k,A/2}^p(B_j)}^p,
	\eeqq
	where
	$$ f_{j,\trans,b}^{\ess}=e^{-i\phi_y(\xi)}\tilde{g}_{\ess,b},$$
	and  for different choices of $b\in\mathcal B$, the corresponding sets $B_j\cap N_{\rho^{1/2+\delta_m}}(Z+b)$ have finite overlaps.

	Thus, one has
	\beqq
	\|Eg_{\ess}\|_{{\rm BL}_{k,A}^p(B_R)}^p\lesssim (\log R)\sum_j \sum_{b\in\mathcal B}\|Ef_{j,\trans,b}^{\ess}\|_{{\rm BL}_{k,A/2}^p(B_j)}^p,
	\eeqq
	and 
	\beqq
	\sum_{b\in\mathcal B}\|\tilde g_{\tn_0,w_0,b}\|_{L^2}^2\lesssim \|g_{\tn_0,w_0}\|_{L^2}^2.
	\eeqq
	Finally, by the equidistribution estimate \eqref{equi}, one has
	\beqq
	\max_{b\in\mathcal B}\|f_{j,\trans,b}^{\ess}\|_{L^2}^2\leq R^{O(\delta_m)}\Big(\f{R^{1/2}}{\rho^{1/2}}\Big)^{-(n+1-m)}\|g_{\ess}\|_{L^2}^2.
	\eeqq
	Now, we may employ an induction on scales argument to complete the proof. By our assumption on $B_j$, we have
	\begin{align*}
		\|Ef_{j,\trans,b}^{\ess}\|_{{\rm BL}_{k,A/2}^p(B_j)}&\leq C(K,\varepsilon, m,D_Z)\rho^{m\varepsilon}\rho^{\delta(\bar A-\log(A/2))}\rho^{-e+1/2}\|f_{j,\trans,b}^{\ess}\|_{L^2}\\
		&\leq C(K,\varepsilon, m,D_Z)R^{\delta}\rho^{m\varepsilon}R^{\delta(\log{\bar {A}}-\log A)}\rho^{-e+1/2}\|f_{j,\trans,b}^{\ess}\|_{L^2}.
	\end{align*}
	Combing the above estimates together, we have
	\begin{align*}
		\|Ef\|_{{\rm BL}_{k,A}^p(B_R)}^p&\lesssim \log R\sum_j\sum_{b\in\mathcal B}\|Ef_{j,\trans,b}^{\ess}\|_{{\rm BL}_{k,A/2}^p(B_j)}^p\\
		&\lesssim R^{O(\delta_m)}\big(C(K,\varepsilon,m,D_Z)\rho^{m\varepsilon}R^{\delta(\log{\bar A}-\log A)}\rho^{-e+1/2})^p\sum_{j,b}\|f_{j,\trans,b}^{\ess}\|_{L^2}^p\\
		&\lesssim R^{O(\delta_m)}\big(C(K,\varepsilon,m,D_Z)\rho^{m\varepsilon} R^{\delta(\log{\bar A}- A )}\rho^{-e+1/2})^p\Big(\f{R^{1/2}}{\rho^{1/2}}\Big)^{-(n+1-m)(p/2-1)}\|f\|_{L^2}^p.
	\end{align*}
	If $p=p(m,k)=2\f{m+k}{m+k-2}$, then
	\beqq
	\rho^{(-e+1/2)p}\Big(\f{R^{1/2}}{\rho^{1/2}}\Big)^{-(n+1-m)(p/2-1)}=R^{(-e+1/2)p},
	\eeqq
	hence,
	\beqq
	\|Ef\|_{{\rm BL}_{k,A}^p(B_R)}\leq C(\varepsilon,D_Z)R^{O(\delta_m)}(R/\rho)^{-mp\varepsilon}\Big(C(K,\varepsilon,m,D_Z)R^{m\varepsilon} R^{\delta(\log{\bar A}-\log A)}R^{-e+1/2}\Big)^p\|f\|_{L^2}^p.
	\eeqq
	Note $R/\rho=R^{O(\delta_{m-1})}$, by choosing $\delta_m\ll \varepsilon\delta_{m-1}$ such that
	\beqq
	C(\varepsilon,D_Z)R^{O(\delta_m)}(R/\rho)^{-mp\varepsilon}\leq 1,
	\eeqq
	then the induction closes and the proof is
	complete.
	
	\section{Parabolic rescaling}\label{paralemma}
	To prove Theorem \ref{theoa2}, we will employ an induction on scales argument. To fulfill the argument, a crucial ingredient is a parabolic rescaling lemma which connects estimates at different scales and facilitates the induction argument. 
	
	For the cone restriction setting in \cite{OW}, one can use the standard Lorentz transformation to tilt the light cone $(\xi, |\xi|)$ into the form $(\xi, \xi_1^2+\cdots+\xi_{n-1}^2/2\xi_n)$ which is well-suited for performing the parabolic rescaling argument, since each vertical slice of this cone is parabolic. However, in our setup of the local smoothing problem for the operator $e^{it\sqrt{-\Delta}}$, the Lorentz transformation is not readily available unless the right-hand side is $L^2$-based.  We can get around this difficulty using the  reductions shown in Section \ref{section-2}. It then suffices to consider phase functions in the class $\mathbf\Phi(R)$, which depends on the scale $R$.

	\begin{lemma}\label{pro2}
		Suppose $\nu$ is a slab of  dimension $1\times K^{-1}\times \cdots \times K^{-1}$ with central line lying in the direction  $\xi_\nu$, then
		\beqq
		\|e^{it\phi_R(D)}f^\nu\|_{L^p(B_R^{n}\times [-R,R])}\leq Q_p(R/K^{2}) K^{-2n(\frac{1}{2}-\frac{1}{p})+\frac{2}{p}-\varepsilon} R^{n(\f{1}{2}-\f{1}{p})+\varepsilon}\|f^\nu\|_{L^p}+{\rm RapDec}(R)\|f\|_{L^p}.
		\eeqq
	\end{lemma}
	
	\begin{proof}
		Without loss of generality, we may assume any $\xi\in\nu$ satisfies
		\beqq
		\Big|\frac{\xi'}{\xi_n}-\xi'_\nu\Big|\leq  \varepsilon_0K^{-1},  \;\; \xi'_\nu \in \R^{n-1}, |\xi'_\nu|\lesssim \varepsilon_0.
		\eeqq
		First, we make a change of variables with respect to $\xi'$ to locate  $\nu$	in a neighborhood of $e_n$
		$$\xi'\rightarrow \xi_n\xi'_\nu+\xi',$$
		correspondingly, the phase function becomes
		\beqq
		\frac{|\xi'|^2}{2\xi_n}+\xi'_\nu\cdot \xi'+\frac{|\xi'_\nu|^2\xi_n}{2}+K^{-4}{\rm E}_R(\xi_n\xi'_\nu+\xi',\xi_n).
		\eeqq 	
		Using the homogeneity  of ${\rm E}_R$ and Taylor's formula, we have 
		\beqq
		\begin{aligned}
			{\rm E}_R(\xi_n\xi'_\nu+\xi',\xi_n)&=\xi_n{\rm E}_R(\xi'_\nu+\frac{\xi'}{\xi_n},1)\\
			&=\xi_n({\rm E}_R(\xi'_\nu,1)+\partial_{\xi'}{\rm E}_R(\xi'_\nu,1)\frac{\xi'}{\xi_n}+\bar{\rm E}_R(\xi)).
		\end{aligned}
		\eeqq	
		Here  $\bar{\rm E}_R(\xi)$ denotes the remainder coming from Taylor's formula. Taking spatial variables into account, the associate phase function reads
		\beqq
		\begin{aligned}
			\big(x'+t\xi'_\nu+tK^{-4}\partial_{\xi'}{\rm E}_R(\xi'_\nu,1)\big)\cdot \xi'+(x_n+t\tfrac{|\xi'_\nu|^2}{2}+tK^{-4}{\rm E}_R(\xi'_\nu,1))\xi_n+t\big(\frac{|\xi'|^2}{2\xi_n}+K^{-4}\bar{\rm E}_R(\xi)\big).
		\end{aligned}
		\eeqq
		Now we perform the change of variables in $(x,t)$ by:
		$$\begin{cases}
			x'+t\xi'_\nu+tK^{-4}\partial_{\xi'}{\rm E}_R(\xi'_\nu,1)&\longrightarrow x',\\
			x_n+t\tfrac{|\xi'_\nu|^2}{2}+tK^{-4}{\rm E}_R(\xi'_\nu,1)&\longrightarrow x_n,\\\quad \quad \quad \quad \qquad 
			t&\longrightarrow t.
		\end{cases}	$$
		Then under the new coordinates,  $B_R^{n}\times [-R,R]$ is transformed into a subset of  $B_{CR}^{n}\times [-CR,CR]$, where $C>0$ is a large absolute constant.	
		
		Now, we perform parabolic rescaling with respect to $x,t ,\xi$ as follows	
		\beqq
		\xi'\longrightarrow K^{-1}\xi', \quad x'\longrightarrow Kx'\,,t\longrightarrow K^2t,
		\eeqq
		the cylinder  $B_{CR}^n\times [-CR,CR]$ is further changed to $B_{CR/K}^{n-1}\times (-CR,CR)\times (-CR/K^2,CR/K^2)$. Now the new phase function  is given by
		\beqq
		\begin{aligned}
			\widetilde{\phi}_{\widetilde{R}}(\xi)=\frac{\xi_1^2+\cdots+\xi_{n-1}^2}{2\xi_n}
			+\widetilde{K}^{-4}{\rm \widetilde E}_{\widetilde{R}}(\xi),\quad \text{where } \widetilde R:=R/K^2,\,\widetilde K:=K_0\widetilde R^{\tilde \delta},
		\end{aligned}
		\eeqq
		where 
		$${\rm \widetilde E}_{\widetilde{R}}(\xi):=\widetilde{K}^{-\tfrac{8\tilde{\delta}}{1-2\tilde\delta}}\xi_nK^2\Big({\rm E}_R(\xi'_\nu+K^{-1}\frac{\xi'}{\xi_n},1)
		-{\rm E}_R(\xi'_\nu,1)-K^{-1}\partial_{\xi'}{\rm E}_R(\xi'_\nu,1)\frac{\xi'}{\xi_n}\Big).$$ 	If we invoke Taylor's formula with the remainder of integral form, we have
		$${\rm \widetilde E}_{\widetilde{R}}(\xi)=\widetilde{K}^{-\tfrac{8\tilde{\delta}}{1-2\tilde\delta}}\Big(\frac{1}{2}\frac{\langle \partial_{\xi'\xi'}^2{\rm E}_R(\xi_\nu,1)\xi',\xi'\rangle}{\xi_n}+K^{-1}\sum_{|\alpha|=3}\frac{3}{\alpha!}\tfrac{(\xi')^\alpha}{\xi_n^2}\int_0^1 (1-t)^2(\partial^\alpha{\rm E}_R)(tK^{-1}\tfrac{\xi'}{\xi_n}+\xi_\nu,1)dt\Big).$$ 
		It is then easy to see that derivatives of ${\rm \widetilde E}_{\widetilde{R}}$ do not blow up in $K$.
		If we set
		\begin{equation*}
			\widehat{g}(\xi)=\widehat{f}^\nu\big(\xi_n \xi'_\nu+K^{-1}\xi',\xi_n\big)
		\end{equation*}
		it is then easy to verify that ${\rm supp}\;{\widehat{g}}\subset {\rm N}_{\varepsilon_0}(e_n)$. 
		By choosing $K_0$ sufficiently large, it is straightforward to check that ${\rm \widetilde E}_{\widetilde{R}}(\xi)$ satisfies condition ${\bf H_3}$, and therefore  $\widetilde{\phi}_{\widetilde{R}}\in\mathbf\Phi(\widetilde R)$.

		Therefore, it suffices to estimate
		\beqq
		e^{it\widetilde{\phi}_{\widetilde{R}}(D)}g:=\int_{\R^{n}} e^{i(x\cdot \xi+t\widetilde{\phi}_{\widetilde{R}}(\xi))}\widetilde a(\xi',\xi_n)\widehat{g}(\xi', \xi_n)\,d \xi.
		\eeqq
		We decompose $B_{CR/K}^{n-1}\times (-CR,CR)\times (-CR/K^2,CR/K^2)$ into a family of finitely overlapping balls of scale $\widetilde R$, i.e.  $$B_{CR/K^2}^{n-1}\times (-CR,CR)\times (-CR/K^2,CR/K^2)\subset \bigcup_y Q_{y,\widetilde R},$$	
		where $y\in \R^{n+1}$ is the center of $Q_{y, \widetilde R}$.
		
		Finally, by a localization argument as in Lemma \ref{lemma-1},  we have
		\begin{align*}
			&\quad \|e^{it\phi_R(D)}f^\nu\|_{L^p(B_R^n\times [-R,R])}^p\lesssim K^{(n+1)-(n-1)p}\sum_y\big\|e^{it\widetilde{\phi}_{\widetilde{R}}(D)}g\|_{L^p(Q_{y,\widetilde{R}})}^p\\
			&\lesssim K^{(n+1)-(n-1)p}\sum_y\|e^{it\widetilde{\phi}_{\widetilde{R}}(D)}\big(\Psi_{B^{n}_{(R/K^2)^{1+\varepsilon/n}}(y)}g\big)\|_{L^p(Q_{y,\widetilde{R}})}^p+{\rm RapDec}(R)\|f\|_{L^p}^p\\
			&\lesssim K^{(n+1)-(n-1)p}\sum_y\|\Psi_{B^{n}_{(R/K^2)^{1+\varepsilon/n}}(y)}g\|_{L^p(\R^n)}^p Q_p^p(R/K^2) \Big(R/K^2\Big)^{np(\frac{1}{2}-\frac{1}{p})}+{\rm RapDec}(R)\|f\|_{L^p}^p\\
			&\lesssim K^{(n+1)-(n-1)p}\|g\|_{L^p(\R^n)}^p Q_p^p(R/K^2) \Big(R/K^2\Big)^{np(\frac{1}{2}-\frac{1}{p})+\varepsilon}+{\rm RapDec}(R)\|f\|_{L^p}^p\\
			&\lesssim K^{-2np(\frac{1}{2}-\frac{1}{p})+2-\varepsilon}\|f^\nu\|_{L^p(\R^n)}^p R^{np(\frac{1}{2}-\frac{1}{p})+\varepsilon}Q_p^p(R/K^2)+{\rm RapDec}(R)\|f\|_{L^p}^p. \end{align*}
	\end{proof}
	
	
	\section{Proof of the main theorem}\label{section-7}
	
	Roughly speaking, the strategy of proving Theorem \ref{theoa2} is to decompose $e^{it\phi_R(D)}f$ into two terms: a ``narrow" term and a ``broad" term. The narrow term comes from the caps of which the normal vectors make a small angle with some $(k-1)$-plane. The broad part comes from the remaining caps. The broad term can be bounded via the broad norm estimate established in Section \ref{section-5}. To bound the narrow term, 
	we need a narrow decoupling theorem.
	
	Let $\delta>0$, $\nu$ be a slab of width $K^{-1}$ defined as usual. We use $\nu_\delta$ to denote the $\delta$ neighborhood of the corresponding slab on the cone defined by 
	\beq
	\nu_\delta:=\{(\eta, \eta_{n+1})\in \R^n\times \R: {\rm dist}((\eta, \eta_{n+1}), (\xi, \phi_R(\xi)))\lesssim \delta, \text{ for some }\xi \in \nu\},
	\eeq
	where $\phi_R$ is a phase function in the class $\mathbf\Phi(R)$.
	
	\begin{theorem}[Narrow decoupling theorem]\label{narrow}
		Let $ k\geq 3$ and $F=\sum_\nu F_\nu$ be a sum over $K^{-1}$ slabs with ${\rm supp}\;{\widehat{F}_\nu}\subset \nu_{K^{-2}}$. Assume that there is a $(k-1)$-dimensional vector space $V$, such that each cap $\nu_{K^{-2}}$ contains a point with normal lying in a $K^{-2}$ neighborhood of $V$. Then for any $\varepsilon>0$,
		\beq
		\|F\|_{L^p(B_{K^2}^{n+1})}\leq C_\varepsilon K^\varepsilon \Big(\sum_{\nu} \|F_\nu\|_{L^p\big(w_{B_{K^2}^{n+1}}\big)}^2\Big)^{1/2}, \; 2\leq p\leq \f{2(k-1)}{k-3}.
		\eeq
	\end{theorem}
	Theorem \ref{narrow} can be deduced from the Theorem 2.3 in \cite{Ha} for the case when the phase function is the circular cone.  By Lorentz transformation, it is also valid for the phase function of the form $(\xi_1^2+\cdots +\xi_{n-1}^2)/2\xi_n$.  Since $\phi_R$ is in the class $\mathbf\Phi(R)$, which is $K^{-4}$-close to the above standard form, Theorem \ref{narrow} is a immediate corollary of Theorem 2.3 in \cite{Ha}.

	Theorem \ref{theoa2} can be deduced from the following proposition. \begin{proposition}\label{proposition-6}
		Let $k\geq 2$. For all $K$, $\varepsilon>0$,  and $\bar{p}(k,n)\leq p\leq 2\frac{k-1}{k-2}$,
		where
		\[\bar{p}(k,n)=
		\begin{cases}
			2\dfrac{(n+1)}{n}\quad &\;\; k=2,\\
			2\dfrac{2n-k+5}{2n-k+3}\quad &\;\; k\geq 3,
		\end{cases}.\]
		If
		\beqq
		\|e^{it\phi_R(D)}f\|_{{\rm BL}_{k,A}^p(C_R^{n+1})}\lesssim_{K,\varepsilon}
		R^{n(\frac{1}{2}-\frac{1}{p})+\varepsilon}\|f\|_{L^p},
		\eeqq
		then we have
		\beqq
		\|e^{it\phi_R(D)}f\|_{L^p(C_R^{n+1})}\lesssim_{\varepsilon}R^{n(\frac{1}{2}-\frac{1}{p})+\varepsilon}\|f\|_{L^p}.
		\eeqq
	\end{proposition}
	\begin{proof}[Proof of Theorem \ref{theoa2}] Recall that it suffices to prove the following estimate 
		\begin{equation}\label{eq:7.3}
			\|e^{it\phi_R(D)}f\|_{L^p(C_R^{n+1})}\lesssim_{\varepsilon}R^{n(\f{1}{2}-\f{1}{p})+\varepsilon}\|f\|_{L^p},\quad {\rm supp}\; \widehat{f}\subset {\rm A}(1)\cap {\rm N}_{\varepsilon_0}(e_n).
		\end{equation}
		By Corollary \ref{cor2} and Proposition \ref{proposition-6}, we obtain \eqref{eq:7.3} if
		\beqq
		p> \min_{2\leq k\leq n+1}\max\Big\{2\frac{n+k+1}{n+k-1},\bar{p}(k,n) \Big\}.
		\eeqq
		In particular, if we choose
		\[
		k=\begin{cases}
			\dfrac{n+5}{2}& \text{if}\;\;n ~\text{\rm is odd},\\
			\dfrac{n+4}{2}& \text{if}\;\;n~ \text{\rm is even},
		\end{cases}
		\]
		then the range of $p$ matches the requirement of Theorem \ref{theoa2}.
	\end{proof}
	\begin{proof}[Proof of Proposition \ref{proposition-6}]We invoke the broad-narrow argument to induct on the scale $R$. The base case is trivial to check. Assume that for any $\varepsilon>0$, $0<R'<R/2$, we have
		$$Q_p(R')\le \bar C_\varepsilon R'^\varepsilon,$$
		we need to prove that
		$$Q_p(R)\le \bar C_\varepsilon R^\varepsilon.$$
		
		For a given ball $B^{n+1}_{K^2}\subset C_{R}^{n+1}$, let $V_1\cdots V_A$ be $(k-1)$-dimensional linear subspaces which achieves the minimum in the definition of the $k$-broad ``norm",  we obtain
		\begin{multline*}
			\int_{B^{n+1}_{K^2}} |e^{it\phi_R(D)}f(x)|^pd xd t\lesssim K^{O(1)}\max_{\tau\notin V_{\ell}}\int_{B^{n+1}_{K^2}} \big|e^{it\phi_R(D)}f^\tau(x)\big|^pd xd t\\+\sum_{\ell=1}^A\int_{B^{n+1}_{K^2}}\Big|\sum_{\tau\in V_{\ell}} e^{it\phi_R(D)}f^\tau(x)\Big|^pd xd t.
		\end{multline*}
		Summing over balls $\{B^{n+1}_{K^2}\}$ yields
		\beq\label{eq-26}
		\begin{aligned}
			\int_{C_{R}^{n+1}} \big|e^{it\phi_R(D)}f(x)\big|^p d x d t\lesssim&  K^{O(1)}\sum_{B^{n+1}_{K^2}\subset C_{R}^{n+1}}\min_{V_1,\cdots V_A}\max_{\tau\notin V_{\ell}}\int_{B^{n+1}_{K^2}} |e^{it\phi_R(D)}f^\tau(x)|^pd x d t\\
			&+\sum_{B^{n+1}_{K^2}\subset C_{R}^{n+1}}\sum_{\ell=1}^A \int_{B^{n+1}_{K^2}}\Big|
			\sum_{\tau\in V_{\ell}} e^{it\phi_R(D)}f^\tau(x)\Big|^pd xd t.
		\end{aligned}
		\eeq
		Invoking Corollary \ref{cor2}, we have
		\begin{equation}\label{add-04}
			\sum_{B^{n+1}_{K^2}\subset C_{R}^{n+1}}\min_{V_1,\cdots V_A}\max_{\tau\notin V_{\ell}}\int_{B_{K^2}^{n+1}} |e^{it\phi_R(D)}f^\tau(x)|^pd xd t\lesssim C(\varepsilon,A,K) R^{np(\frac{1}{2}-\frac{1}{p})+ \varepsilon p/2}\|f\|_{L^p}^p.
		\end{equation}
		
		Next, we use  Theorem \ref{narrow} and  Lemma \ref{pro2}
		to estimate the contribution of the second term in the right-hand side of \eqref{eq-26}.  
		It follows from Theorem \ref{narrow} that for any $\delta>0$
		\beq
		\begin{aligned}
			\sum_{\ell=1}^A\int_{B^{n+1}_{K^2}}\Big|\sum_{\tau\in V_{\ell}}& e^{it\phi_R(D)}f^\tau(x)\Big|^pd x d t
			\\&\lesssim C(\delta,A) K^\delta \max\{1,K^{(k-3)(\frac 12-\frac1p)p}\}\sum_\tau\int_{\R^{n+1}} w_{B^{n+1}_{K^2}}\big|e^{it\phi_R(D)}f^\tau(x)\big|^pd xd t,
		\end{aligned}
		\eeq
		where we have used the fact that 
		$$\#\{\tau: \tau \in V_\ell\}\lesssim \max\{1,K^{k-3}\}. $$
		Summing over $B^{n+1}_{K^2}$ in both sides of the above inequality,  we obtain
		\begin{align}
			\sum_{B^{n+1}_{K^2}\subset C_{R}^{n+1}}\sum_{\ell=1}^A&\int_{B^{n+1}_{K^2}}\Big|\sum_{\tau\in V_{\ell}} e^{it\phi_R(D)}f^\tau(x)\Big|^pd x d t\nonumber\\ \lesssim& C(\delta,A)  K^\delta \max\{1,K^{(k-3)(\frac 12-\frac1p)p}\}\sum_{\tau}\int_{\mathbb{R}^{n+1}} w_{C_{R}^{n+1}} \Big|
			e^{it\phi_R(D)}f^\tau(x)\Big|^pd xd t.\label{add-05}
		\end{align}
		Using the rapidly decaying property of the weight function, we have
		\beqq
		\int_{\mathbb{R}^{n+1}} w_{C_{R}^{n+1}} \Big|
		e^{it\phi_R(D)}f^\tau(x)\Big|^pd xd t\leq \int_{C_{R^{1+\delta}}^{n+1}}\Big|
		e^{it\phi_R(D)}f^\tau(x)\Big|^pd xd t+{\rm RapDec}(R)\|f\|_{p}^p.
		\eeqq
		For $\varepsilon_1>0$, we see that by  Lemma \ref{pro2} 
		
		\begin{align*}
			&\int_{C_{R^{1+\delta}}^{n+1}}\Big|
			e^{it\phi_R(D)}f^\tau(x)\Big|^pd xd t\leq \sum_{C_R^{n+1}\subset C_{R^{1+\delta}}^{n+1}}\int_{C_{R}^{n+1}}\Big|
			e^{it\phi_R(D)}f^\tau(x)\Big|^pd xd t\\
			\lesssim_{\varepsilon_1}& K^{-2 n(\frac{1}{2}-\frac{1}{p})p+2-\varepsilon_1}Q^p_p\Big(\f{R}{K^2}\Big)R^{np(\f{1}{2}-\f{1}{p})+\varepsilon_1+(n+1)\delta}\big\|f^\tau\big\|_p^p+{\rm RapDec}(R)\|f\|_{p}^p.
		\end{align*}
		Summing over $\tau$ and noting that
		\beqq
		\sum_{\tau}\|f^\tau\|_p^p\leq C \|f\|_{p}^p, \quad\text{ for} \;\;2\leq p\leq \infty,
		\eeqq
		we have
		\begin{align}
			&\sum_\tau \int_{\mathbb{R}^{n}} w_{C_{R}^{n+1}} \Big|
			e^{it\phi_R(D)}f^\tau(x)\Big|^pd xd t\nonumber\\
			\lesssim_{\varepsilon_1}& K^{-2 n(\frac{1}{2}-\frac{1}{p})p+2-\varepsilon_1}Q^p_p\Big(\f{R}{K^2}\Big)R^{np(\f{1}{2}-\f{1}{p})+\varepsilon_1+(n+1)\delta}\big\|f\big\|_p^p+{\rm RapDec}(R)\|f\|_{p}^p.\label{add-06}
		\end{align}
		Collecting the estimates \eqref{add-04}-\eqref{add-06} and inserting them into \eqref{eq-26},  we obtain 
		\begin{align*}
			\int_{C_{R}^{n+1}}|e^{it\phi_R(D)}f(x)|^p&d xd t\leq C(\varepsilon,A,K)R^{np(\frac{1}{2}-\frac{1}{p})+\varepsilon p/2}\|f\|_{L^p}^p\\&+C(\delta,\varepsilon_1,A)K^\delta R^{np(\f{1}{2}-\f{1}{p})+\varepsilon_1+(n+1)\delta}K^{-e(p,k,n)-\varepsilon_1}Q^p_p\Big(\f{R}{K^2}\Big)\|f\|_{L^p}^p,
		\end{align*}
		where
		$$e(p,k,n):=\max\Big\{2n(\f{1}{2}-\f{1}{p})p-2, 2n(\f{1}{2}-\f{1}{p})p-2-(k-3)(\f{1}{2}-\f{1}{p})p\Big\}.$$
		Note that $e(p,k,n)\ge0,$
		if 
		\[p\ge
		\begin{cases}
			\dfrac{2(n+1)}{n}\quad &\;\; k=2,\\
			2\dfrac{2n-k+5}{2n-k+3}\quad &\;\; k\geq 3,
		\end{cases},\]
		therefore by the definition of $Q_p(R)$, we have
		\beqq
		Q^p_p (R)\leq C(\varepsilon,A,K)R^{\frac{\varepsilon}2p}+C(\delta,\varepsilon_1,A)K^\delta R^{(n+1)\delta+\varepsilon_1}K^{-\varepsilon_1}Q^p_p\Big(\f{R}{K^2}\Big).
		\eeqq
		Invoking the induction hypothesis, we have
		\beqq
		Q^p_p (R)\leq C(\varepsilon,A,K)R^{\frac{\varepsilon}2p}+\bar C_\varepsilon R^\varepsilon C(\delta,\varepsilon_1,A) R^{(n+1)\delta+\varepsilon_1}K^{\delta-\varepsilon_1-2\varepsilon}.
		\eeqq
		The first term is harmless. If we can choose suitable $\delta,\tilde\delta,  \varepsilon_1$ such that 
		\beq \label{want} C(\delta,\varepsilon_1,A) R^{(n+1)\delta+\varepsilon_1}K^{\delta-\varepsilon_1-2\varepsilon}<\f 12,\eeq
		then the induction closes. Fixing $p>\bar{p}(k,n)$, recall that $K=K_0R^{\tilde{\delta}}$, the right hand side of\eqref{want} is bounded above by
		\beq\label{LHS}
		C(\delta,\varepsilon_1,A) R^{(n+1)\delta+\varepsilon_1+\tilde\delta(\delta-\varepsilon_1-2\varepsilon)}
		\eeq
		
		By choosing 
		$$\delta=\varepsilon_1=\frac{\varepsilon^2}{10n},\;\;  \;\;\tilde\delta=\varepsilon,$$ we see that \eqref{LHS} is bounded above by
		\beqq
		C(\varepsilon,A) R^{-\varepsilon^2},
		\eeqq
		which is less than $\f 12$ if $R\gg 1$, the induction closes.
	\end{proof}
	
	Lastly, recall that our extension operator $E$ is defined with respect to a cone of the form $(\xi,\phi(\xi))$ with $\phi\in\mathbf\Phi$. After reducing to the smaller class $\mathbf \Phi(R)$, one may run the argument of Ou--Wang \cite{OW} to see that our $k$-broad and narrow decoupling bounds imply the following generalization of Theorem 1 in \cite{OW} to such a general cone.
	\begin{theorem}\label{restriction}
		For any $n\geq 2$ and
		\[
		p>\begin{cases}
			4& \text{if}\;\; n=2,\\
			2\dfrac{3n+4}{3n}& \text{if}\;\;n>2 ~\text{\rm is even},\\
			2\dfrac{3n+3}{3n-1}& \text{if}\;\;n>2~ \text{\rm is odd},
		\end{cases}
		\]
		then
		\beqq
		\|Ef\|_{L^p(\R^{n+1})}\leq C_p\|f\|_{L^p({\rm A}(1))}.
		\eeqq
	\end{theorem}

	\bibliographystyle{amsplain}

\end{document}